\documentclass[imslayout,noinfoline]{imsart}

\RequirePackage[OT1]{fontenc}
\RequirePackage{amsthm,amsmath,natbib}

\usepackage{bm} %% Bold Math
\usepackage{graphicx,psfrag}
\usepackage{epstopdf}
\usepackage{verbatim}
\usepackage{upref}
\usepackage{multirow}
\usepackage{color}
\usepackage[labelformat=empty]{subfig}
\usepackage{amssymb,epsf, eucal}
\usepackage{multirow}
\usepackage{color}
\startlocaldefs
\newtheorem{theorem}{Theorem}

\newtheorem{lemma}{Lemma}
\newtheorem{corollary}[theorem]{Corollary}

\newcommand*{\bs}{\boldsymbol}

\newcommand\ST{\rule[-1em]{0pt}{2.5em}}
\endlocaldefs

\begin{document}

\begin{frontmatter}
\title{On the consistent separation of scale and variance for Gaussian random fields}
\runtitle{Separation of scale and variance}
\begin{aug}
\author{\fnms{Ethan} \snm{Anderes}\ead[label=e1]{anderes@stat.ucdavis.edu}}
%\author{\fnms{Ethan} \snm{Anderes}\thanksref{t1}\ead[label=e1]{anderes@stat.ucdavis.edu}}
%\thankstext{t1}{Supported in part by an NSF Postdoctoral Fellowship  DMS-0503227.}
\runauthor{E. Anderes}
\affiliation{University of California at Davis}
\address{Department of Statistics\\University of California at Davis\\4214 Math Sci. Bldg.\\ One Shields Ave.\\Davis, CA 95616}
\end{aug}

\begin{abstract}
We present fixed domain asymptotic results that establish consistent estimates of the variance and scale parameters for a Gaussian random field with a geometric anisotropic Mat\'ern autocovariance in dimension $d>4$. When $d<4$ this is impossible due to the mutual absolute continuity of Mat\'ern Gaussian random fields with different scale and variance (see Zhang \cite{zhang:2004}).  
Informally, when $d>4$, we show that one can estimate the coefficient on the principle irregular term accurately enough to get a consistent estimate of the coefficient on the second irregular term. These two coefficients can then be used to separate the  scale and variance.  We extend our results to the general problem of estimating a variance and geometric anisotropy for more general  autocovariance functions. 
Our results illustrate the interaction between the accuracy of estimation, the smoothness of the random field, the dimension of the observation space, and the number of increments  used for estimation. As a corollary, our results establish the orthogonality of Mat\'ern Gaussian random fields with different parameters when $d>4$. The case $d=4$ is still open.
\end{abstract}

%60G60 Random fields
%62M30 Spatial processes
%62M40 Random fields; image analysis

%Nonparametric inference
%62G05 Estimation



\end{frontmatter}

%%%%%%%%%%%%%%%%%%%%%%%%%%%%%%%%%%%%%%%%%
\section{Introduction}
%%%%%%%%%%%%%%%%%%%%%%%%%%%%%%%%%%%%%
A common situation in spatial statistics is when one has observations on a single realization of a random field $Y$ at a large number of spatial points $\bs t_1,\bs t_2, \ldots$ within some bounded region $\Omega\subset \Bbb R^d$. One is then is faced with the problem of predicting some quantity that depends on $Y$ at unobserved points in $\Omega$. For example, one may want to predict $\int_\Omega Y(\bs t) d\bs t$ or the derivative $Y^{\prime}(\bs t_0)$ where $\bs t_0$ is an unobserved point in $\Omega$. A common technique is to first estimate the covariance structure of $Y$, then predict using the estimated covariance. Typically, fully nonparametric estimation of the covariance is difficult since the observations are from one realization of the random field.  In this case, it is common to consider a class of covariance structures indexed by a finite number of parameters which are then estimated from the observations (see \cite{cressie:book} or \cite{chiles:book} for an introduction to spatial statistical techniques).

Two common parameters found in many covariance models are an overall scale $\alpha$ and an overall variance $\sigma^2$. The simplest example of this model stipulates that the random field $Y$ is a scale and amplitude chance by an unknown $\alpha$ and $\sigma$ of a known random field $Z$. In particular, for a spatial domain $\Omega\subset \Bbb R^d$, $Y$ is modeled as
\begin{equation} 
\label{FirstModel}
 \{ Y(\bs t)\colon \bs t\in \Omega\}\overset {\mathcal D}= \{\sigma Z(\alpha \bs t)\colon \bs t\in\Omega\} 
 \end{equation}
where $\overset{\mathcal D}=$ denotes equality of the finite dimensional distributions. In this case, $\sigma$ is an overall amplitude (in units of $Y$) and $\alpha$ is an overall spatial scale (in units of $\bs t$).
For a nice discussion of the roll of $\alpha$ and $\sigma$ in the Mat\'ern autocovariance see Section 6.5 in \cite{stein:book}.

A fundamental question is whether or not $\alpha$ and $\sigma$ are consistently estimable when the number of the observations in $\Omega$ grows to infinity.  Indeed, the answer is no in general.   This is immediate from the existence of self similar random fields that satisfy  $\{ Z(\alpha \bs t)\colon \bs t\in\Omega\}\overset{\mathcal D}=\{\alpha^{\nu}Z(\bs t)\colon \bs t\in \Omega\}$ for any $\alpha>0$ where  $\nu$ is a fixed constant. For these self-similar processes, any two pairs $(\sigma_1, \alpha_1)$ and $(\sigma_2,\alpha_2)$ that satisfy $\sigma_1^2\alpha_1^{2\nu}=\sigma_2^2 \alpha_2^{2\nu}$ give the same model in (\ref{FirstModel}).  
This problem can also be present when $Z$ is not self similar. For example, suppose $Z$ is an isotropic Ornstein-Uhlenbeck process in dimension $d\leq 3$ (see Figure \ref{fig1}). In this case, if $\sigma_1^2\alpha_1=\sigma_2^2 \alpha_2$ (i.e. $\nu=1/2$) the two models for $Y$ yield mutually absolutely continuous measures (when $d=1$ see \cite{Ibragimov}, \cite{ying:1991},  when $d=2,3$ see \cite{zhang:2004}, \cite{stein:book}) and therefore are impossible to discern with probability one when observing one realization of $Y$.   
We shall see, however, that in some cases it is possible to consistently estimate $\alpha$ and $\sigma$. Moreover, it will depend on dimension: typically the larger the dimension the more information there is to  separate $\sigma$  from $\alpha$.  Before we continue, we mention the work of Stein (see \cite{stein:Asymp:88},\cite{stein:Unif:1990}) which establishes that even if two models are mutually absolutely continuous, using the wrong model to make predictions may still yield asymptotically optimal estimates.  In fact, this phenomenon can also occur for orthogonal measures when restricting to predictors that are linear combinations of the observations (see \cite{stein:1993}).
%but now for linear predictions using the wrong model my still be asymptotically optimal over the class of all linear estimators (see \cite{stein:1993}).

To understand the condition $\sigma_1^2\alpha_1^{2\nu}=\sigma_2^2 \alpha_2^{2\nu}$  one can look at what is called the principle irregular term of the autocovariance function (see \cite{stein:book}). Suppose, for exposition, that there exist constants $\delta_2>\delta_1>0$ such that the covariance structure of $Z$ satisfies
\begin{equation}
\label{PrinceTerm} 
\text{cov}(Z(\bs t+\bs h),Z(\bs t))\approx  c_1 |\bs h|^{\delta_1}+ c_2|\bs h|^{\delta_2}+p(|\bs h|),\quad\text{as $|\bs h|\rightarrow 0$} 
\end{equation}
where $p$ is an even polynomial and both $\delta_1, \delta_2$ are not even integers. This model is not as restrictive as it seems and includes the Ornstein-Uhlenbeck process, the exponential autocovariance function $e^{-|\bs s-\bs t|^{\delta_1}}$ and the Mat\'ern autocovariance function (see below).
The term $c_1|\bs h|^{\delta_1}$ is often referred to as the principle irregular term and is instrumental in determining the smoothness of $Z$.  The second term, $c_2|\bs h|^{\delta_2}$, is less influential but  can have an observable effect depending on dimension and the magnitude of $\delta_2-\delta_1$.
Now, if we model $Y$ by (\ref{FirstModel}) and (\ref{PrinceTerm}) we get
\begin{equation}
\label{prince2}
\text{cov}(Y(\bs t+\bs h),Y(\bs t))\approx  c_1\sigma^2\alpha^{\delta_1} |\bs h|^{\delta_1}+ c_2\sigma^2\alpha^{\delta_2}|\bs h|^{\delta_2}+\tilde p(|\bs h|),\quad\text{as $|\bs h|\rightarrow 0$.} 
\end{equation}
Therefore for two pairs of parameters   $(\sigma_1,\alpha_1)$ and $(\sigma_2,\alpha_2)$, the condition $\sigma_1^2\alpha_1^{\delta_1}=\sigma_2^2 \alpha_2^{\delta_1}$ ensures that the covariance models for $Y$  have the same principle irregular term.
%Notice that under the model (\ref{FirstModel}) and (\ref{PrinceTerm}), the principle irregular term of $Y$ is $c_1 \sigma^2 \alpha^{\delta_1}|\bs h|^{\delta_1}$. Indeed, any two pairs of parameters   $(\sigma_1,\alpha_1)$ and $(\sigma_2,\alpha_2)$ such that   $\sigma_1^2\alpha_1^{\delta_1}=\sigma_2^2 \alpha_2^{\delta_1}$  have models for $Y$ in (\ref{FirstModel}) with the same principle irregular term. 
This explains the importance of  the quantity $\sigma^2\alpha^{\delta_1}$. 
In addition, if one can estimate both coefficients $c_1\sigma^2\alpha^{\delta_1}$ and  $c_2 \sigma^2 \alpha^{\delta_2}$ then it is possible to get separate estimates of $\sigma$ and $\alpha$. 
% Now if one can estimate the coefficient on the second order term and recover both $c_1 \sigma^2 \alpha^{2\delta_1}$ and $c_1 \sigma^2 \alpha^{2\delta_2}$  one can seperate  $\sigma$ from $\alpha$.
In what follows we develop consistent estimators of these two coefficients which allow consistent estimation of $\sigma$ and $\alpha$.

\begin{figure}
\centering
\subfloat[$Z(2t)$]{\includegraphics[height=1.5in]{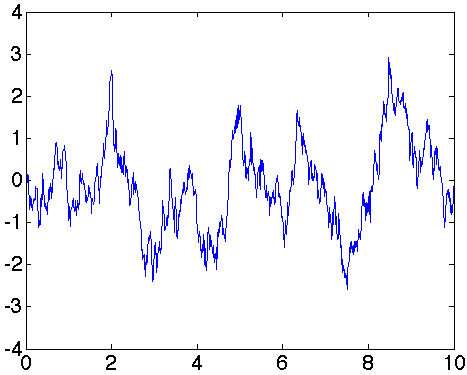}}
\qquad
\subfloat[$\sqrt{2}Z(t)$]{\includegraphics[height=1.5in]{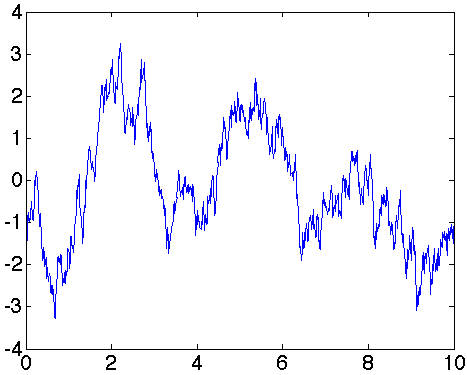}}
\caption{Independent simulations of $Z(2 t)$ and $\sqrt{2} Z(t)$, observed on a dense grid  in $[0,10]$, where $Z$ is the Ornstein-Uhlenbeck process with covariance structure $\text{cov}(Z(s), Z(t))=e^{-|s-t|}$. In $1
$, $2$ and $3$ dimensions these two processes (isotropically extended) are mutually absolutely continuous and therefore cannot be consistently distinguished under fixed domain asymptotics. Our results establish that when the dimension is greater than 4 one can distinguish the two with probability one under fixed domain asymptotics.}
\label{fig1}
\end{figure}

The majority of this paper focuses on the case when $Z$ is a mean zero, isotropic Gaussian random field which has a Mat\'ern autocovariance.   The reasons are twofold. First, the Mat\'ern autocovariance has been used extensively in spatial statistics so that results on the Mat\'ern autocovariance are of intrinsic interest alone. The second reason is that once one establishes the results for the Mat\'ern it is relatively easy to see how to extend to other covariance functions. In Section  \ref{BeyondB}, we give two examples that illustrate these extensions. 
Our Mat\'ern assumption stipulates the existence of a known $\nu>0$ such that
\begin{equation}
\label{MMMater}
\text{cov}(Z(\bs s), Z(\bs t))= \frac{|\bs s-\bs t|^\nu \mathcal K_\nu(|\bs s-\bs t|)}{2^{\nu-1}\Gamma(\nu)} 
\end{equation}
for all $\bs s,\bs t\in \Omega\subset \Bbb R^d$ where $|\cdot|$ denotes Euclidean distance and $\mathcal K_\mu$ is the modified Bessel function of the second kind of order $\nu>0$  (see \cite{AShandbook}). The parameter $\nu$ controls the mean square smoothness of the process: larger $\nu$ corresponds to smoother $Z$. The flexibility provided by the smoothness parameter $\nu$ along with the fact that it is positive definite in any dimension leads to its widespread use in spatial statistics.  

 In what follows we extend the basic model (\ref{FirstModel}) to the case when there is an unknown invertible  matrix $M$ with determinant 1 (this class of matrices we denote by $SL(d,\Bbb R)$) so that
\begin{equation} 
\label{SecondModel}
 \{ Y(\bs t)\colon \bs t\in \Omega\}\overset {\mathcal D}= \{\sigma Z(\alpha M \bs t)\colon \bs t\in\Omega\}.
 \end{equation}
The matrix $M$ is called a geometric anisotropy and is used to model a directional sheer of $Z$. The assumption that $\det M=1$ removes identifiability problems with the overall scale parameter $\alpha$.  In Section \ref{Geoo}, we construct estimates of $\sigma^2\alpha^{2\nu}$, $M$  and $\alpha$. We show that the estimates of $\sigma^2\alpha^{2\nu}$ and $M$ are strongly consistent in any dimension and the estimate of $\alpha$ is strongly consistent when $d>4$.

There is a fair amount of literature on estimating $\sigma^2\alpha^{2\nu}$ for the Mat\'ern autocovariance. In 1991, Ying \cite{ying:1991} established strong consistency and the asymptotic distribution of the maximum likelihood estimate of $\sigma^2\alpha^{2\nu}$ for the Ornstein-Uhlenbeck process when $d=1$ (which has a Mat\'ern autocovariance for $\nu=1/2$).
In 2004, Zhang \cite{zhang:2004} established that the maximum likelihood estimate of $\sigma^2\alpha^{2\nu}$ (obtained by fixing $\alpha$ and $\nu$) is strongly consistent when $d\leq 3$. In related work, Loh \cite{loh2005} shows that maximum likelihood estimates of scale and variance parameters in a non-isotropic multiplicative Mat\'ern model are consistent when $\nu=3/2$ (similar results for the Gaussian autocovariance model can be found in \cite{loh2000}). In Section 6.7 of \cite{stein:book}, Stein derives asymptotic properties of the maximum likelihood estimates of $\alpha$, $\sigma$ and $\nu$ for a periodic version of the Mat\'ern random field. For this periodic random field  all the parameters are consistently estimable when $d\geq 4$. Our results confirm these findings for $\alpha$ and $\sigma$ with the non-periodic Mat\'ern when $d>4$. The case $d=4$ is still open.
%Stein "Equiv of Gaussian", \cite{stein:equiv}.
%H. Zhang's Biometrica paper. \cite{zhang:bio05}

Recent work by Kaufman et al.\! \cite{kaufman:taper} and Du et al.\! \cite{Du2009} studies maximum likelihood estimates of $\sigma^2 \alpha^{2\nu}$ using a tapered Mat\'ern autocovariance when $d\leq 3$. The advantage gained by tapering is a reduction of the computational load for computing the likelihood and for computing kriging estimates.
We will see that our estimates of the same quantity, $\sigma^2 \alpha^{2\nu}$, yield strongly consistent estimates in any dimension which are \lq\lq root n" consistent and are easily computed with no maximization required.  However, our estimates depend on the grid format of the observations whereas the maximum likelihood estimates are not confined to such restrictions.  We also expect some loss of efficiency in our estimates as compared to the MLE. We hope that there is potential to combine the two estimation methods using a one-Newton-step tapered likelihood adjustment to the increment based estimate. Since our results can be easily extended, by a Lindeberg-Feller argument,  to obtain the asymptotic normality of $\widehat{\sigma^2\alpha^{2\nu}}$ when $d\leq 3$, we believe this has the potential  to mitigate any loss of efficiency and reduce the computational load for the maximum likelihood estimate. 

Finally we mention the long tradition of using squared increments to estimate properties of random fields, beginning with the quadratic variation theorem of L\'evy in 1940 (\cite{levy:qv}).  For example, increments have been used in  \cite{IstasLang:qv} and \cite{BenassiCohen98:qv} for identification of a local fractional index and in  \cite{CohenGuyon:Singularity} to identify the singularity function of a fractional process. In \cite{anderesSourav} they  are used to estimate a deformation of an isotropic random field.
For more results on the convergence of quadratic variations see, \cite{baxter:qv},  \cite{gladyshev:qv},  \cite{dudley:qv}, \cite{klein:qv}, \cite{berman:qv}, \cite{strait:qv},\cite{AdlerPyke:qv}, \cite{BenassiCohen98:qv}, \cite{GuyonLeon:qv}, \cite{CohenGuyon:Identification},  \cite{IstasLang:qv}.

%\Comment{
%Here is where we introduce all this stuff about the Matern.... Especially mention the corrolarry of estimating the quantity $\sigma^2 \alpha^{2\nu}$; interpreting the relationship with dimension, smoothness and number of increments; and geometric anisoptic estimation. Also mention the tapering results.
%Also mention that it resolves the issue of orthogonality for the Matern model in $d>4$ but leaves open $d=4$ where more delicate methods seem necessary. You might also want to show (more likely you should ellude to it in the discussion) that when one has a process with generalized autocovariance function $c_1|t|^{\alpha_1} + c_2|t|^{\alpha_2}+\cdots+c_p|t|^{\alpha_p}$, under what conditions can you consistently estimate $c_1,\ldots,c_p$?
%You may also want to mention the implications for prediction. If one uses a bad estimate of $\rho$ when your in higher dimensions you might have problems with your interpolations and your estimates of the mse of those interpolations. Obvisously you will want to reference Stein here. Advantages...can get results...computationally easy, do not need to estimate $\alpha$ or use a guess. Disadvantages...needs some type of a regular grid.
%}

%%%%%%%%%%%%%%%%%%%%%%%%%%
\section{The geometric anisotropic Mat\'ern class}
%%%%%%%%%%%%%%%%%%%%%%%
\label{Geoo}
In this section we construct estimates of $\sigma^2\alpha^{2\nu}$, $M$  and $\alpha$ using increments of   $Y$ observed on a dense grid within $\Omega$. Using fixed domain asymptotics, we establish consistency of our estimates  under assumptions (\ref{MMMater}) and (\ref{SecondModel}) and  provide bounds on the rate of variance decay as it depends on the number of increments used, the dimension of $\Omega$ and the smoothness of $Y$ measured by $\nu$. These results will hold in any dimension. However, when the dimension is large enough ($d>4$), the second term in (\ref{prince2}) is influential enough  so that  $\alpha$ can be estimated consistently. 
%As a corollary, this establishes the previously unknown result that Gaussian random fields with different Mat\'ern parameters generate orthogonal measures when $d>4$. The final case when $d=4$ is still open.

If the observation region $\Omega$ is an open subset of $\Bbb R^d$ and the random field $Y$ is modeled by (\ref{MMMater}) and (\ref{SecondModel}), then $Y$ is said to be a $d$-dimensional geometric anisotropic Mat\'ern random field with parameters $(\sigma, \alpha,\nu,M)$. In this case, the covariance structure of $Y$ is $\text{cov}(Y(\bs s),Y(\bs t))=K(|M\bs s-M\bs t|)$,
 where $K$ is defined as 
\begin{equation}
\label{DefineK}
K(t)\triangleq\frac{\sigma^2 (\alpha t)^\nu}{\Gamma(\nu) 2^{\nu-1}} \mathcal K_\nu(\alpha t)  
\end{equation}
 for $t>0$ and $K(0)\triangleq \lim_{t\downarrow 0} K(t)=\sigma^2$. The function
 $\mathcal K_\nu$ is the modified Bessel function of the second kind of order $\nu>0$.  Since $|M\bs s-M\bs t|= |OM\bs s-OM \bs t|$ for any orthogonal matrix $O$, one can only identify $M$ up to left multiplication by an orthogonal matrix. To remove this identifiability problem we suppose that $M\in SL(d,\Bbb R)/SO(d,\Bbb R)$ where $SO(n, \Bbb R)$ denotes the orthogonal matrices in $SL(d,\Bbb R)$. In the theorems below, we write $M_1=_{SL/SO} M_2$ to mean that  there exists a $O\in SO(n, \Bbb R)$ such that $M_1=O M_2$, and similarly for $M_1\neq_{SL/SO} M_2$.
Operationally, however, we estimate a representer of the cosets in $SL(d,\Bbb R)/SO(d,\Bbb R)$ given by the upper triangular matrices which have positive diagonal elements and determinant 1 (that this is a representer follows from the QR factorization, see \cite{HornJohnson:book}).

  As discussed in the introduction, the principle irregular term is important in determining the sample path properties of the random field $Y$. The principle irregular term for the Mat\'ern covariance function is  
\[ G_\nu(t)\triangleq\begin{cases} 
\ST\dfrac{(-1)^{\nu+1}}{2^{2\nu-1}\Gamma(\nu)\Gamma(\nu+1)} t^{2\nu} \log t , & \text{if $\nu\in \Bbb Z$;} \\ 
\ST\dfrac{-\pi}{2^{2\nu}\sin(\nu\pi)\Gamma(\nu)\Gamma(\nu+1)}t^{2\nu}, & \text{otherwise.}
 \end{cases} \]
   where $G_{\nu}(0)$ is defined to be $0$.
Moreover,
\begin{align}
\label{PrinceTerms2} 
\text{cov} (Y(\bs t+\bs h),Y(\bs t))= \sigma^2 G_\nu(|\alpha M\bs h|) &-\nu \sigma^2 G_{\nu+1}(|\alpha M\bs h|)+ \epsilon(|\alpha M\bs h|) 
\end{align}
where $\epsilon( h)= \sigma^2 p(|h|)+ o(G_{\nu+1}(| h|))$ as $|h|\rightarrow 0$ and $p$ is an even polynomial. Notice that when $M$ is the identity matrix and $\nu\notin \Bbb Z$, this gives the expansion (\ref{prince2}) so that $\sigma^2 G_\nu(|\alpha \bs h|)$ is the first principle irregular term and $-\nu \sigma^2 G_{\nu+1}(|\alpha \bs h|)$ is the second term.

%%%%%%%%%%%%%%%%%%%
\subsection{Estimating $\sigma^2\alpha^{2\nu}$ and $M$ in any dimension}
\label{AnyDim}
%%%%%%%%%%%%%%%%%%%%%%%%%%%%%

Let $\Omega$ be a bounded, open subset of $\Bbb R^d$ and let $\Omega_n\triangleq \Omega\cap \{ \Bbb Z^d/n \}$.  The idea is that we will be observing $Y$ on a region, just a bit larger than $\Omega_n$, so that we can form the $m^\text{th}$ order increments of $Y$ on $\Omega_n$. These will then be used to estimate $M$ and $\sigma^2\alpha^{2\nu}$ in any dimension and additionally  $\alpha$, in dimension $d>4$.

 For a fixed nonzero vector $\bs h\in \Bbb R^d$ define the increment in the direction $\bs h$ by $\Delta_{\bs h} Y(\bs t)\triangleq Y(\bs t+\bs h)-Y(\bs t)$ and the $m^\text{th}$ iterated directional increment  $\Delta^m_{\bs h}Y(\bs t)\triangleq\Delta_{\bs h}\Delta^{m-1}_{\bs h}Y(\bs t)$. The following lemma establishes the relationship between the variance of these increments and the terms in (\ref{PrinceTerms2}) when the number of increments is sufficiently large. 

%%%%%%%%%%%%%%%%%%
\begin{lemma}  
\label{FirstLemma}
Let $Y$ be a mean zero, geometric anisotropic $d$-dimensional Mat\'ern Gaussian random field with parameters $(\sigma, \alpha, \nu, M)$.  If $m$ is a positive integer such that $m>\nu+1$ and $\bs h\in \Bbb R^d$ is a non-zero vector, then
\begin{equation}
\label{eept}
\mathsf E(\Delta_{\bs h/n}^{m} Y(\bs t))^2 = \frac{a_\nu^{m}}{n^{2\nu}}+  \frac{b_\nu^{m}}{n^{2\nu+2}}+o(n^{-2\nu-2})
\end{equation}
 as $n\rightarrow \infty$ where 
\begin{align}
\label{AAA}
a_\nu^m &\triangleq \sigma^2\alpha^{2\nu}  |M\bs h|^{2\nu} \sum_{i,j=0}^m (-1)^{i+j}{m\choose i}{m\choose j} G_\nu(|i-j|) \\
\label{BBB}
 b_\nu^{m} &\triangleq \sigma^2\alpha^{2\nu+2}  |M\bs h|^{2\nu+2} \sum_{i,j=0}^{m} (-1)^{i+j}{m\choose i}{m\choose j} (-\nu)G_{\nu+1}(|i-j|). 
 \end{align}
\end{lemma}

Now we are in a position to  estimate the coefficient $a_\nu^m$. 
Let $\#\Omega_n$ denote the cardinality of the finite set $\Omega_n\triangleq \Omega\cap\{\Bbb Z^d/n \}$ and define
\begin{equation}
\label{DefofQQ}
 Q^m_n\triangleq  \frac{1}{\#\Omega_n} \sum_{\bs j\in\Omega_n} n^{2\nu}(\Delta_{\bs h/n}^m Y(\bs j))^2
 \end{equation}
Notice that by equation (\ref{eept}), $\mathsf E Q_n^m\rightarrow a_\nu^m$ as $n\rightarrow \infty$. In addition, since $Q_n^m$ is itself an average, one might hope that $Q_n^m$ converges to $a_\nu^m$. The following theorem shows that, indeed, this is the case. In addition, the theorem quantifies the decay of the variance of $Q_n^m$ as a function of the number of increments, the smoothness of the random field $Y$  and the dimension of the domain. The heuristic is that when the number of increments $m$ is large enough,  there is sufficient decorrelation of the summands of $Q_n^m$ to guarantee convergence as $n\rightarrow \infty$. Generally, more increments leads to more spatial decorrelation and hence a reduction in variance. However, this only holds up to a point, after which taking more increments no longer effects the rate of variance decay. Finally, the higher the dimension, the more increments one needs to take to get the best rate.

%-------------------------------
\begin{theorem}  
\label{FristThem}
 Let $Y$ be a mean zero, geometric anisotropic $d$-dimensional Mat\'ern Gaussian random field with parameters $(\sigma, \alpha, \nu, M)$ and let $\Omega$ be a bounded, open subset of $\Bbb R^d$. If $m>\nu$ then 
\begin{equation}
Q_n^m\rightarrow a_\nu^m , \quad w.p.1 
\end{equation}
as $n\rightarrow \infty$.
Moreover, there exists a constant $c>0$ such that 
\[
\text{var}\, Q_n^m \leq \begin{cases}
c\,n^{4(\nu-m)},& \text{if $4(\nu-m)>-d$;} \\
c\,n^{-d}\log n,& \text{if $4(\nu-m)=-d$;} \\
c\,n^{-d}, &  \text{if $4(\nu-m)<-d$}
\end{cases}
\]
for all sufficiently large $n$.
\end{theorem}

The above theorem establishes that $Q_n^m$ consistently estimates $a_\nu^m$ (which depends on $\bs h$). Now we show how these estimates can be used to recover $M$ and $\sigma^2 \alpha^{2\nu}$.
As was mentioned above, we suppose $M$ is upper triangular with determinant one and positive diagonal elements. After renormalizing by known constants, the values of $a_\nu^m$ allow us to consistently estimate $|\tilde M \bs h|^2$ where $\tilde M\triangleq \sigma^{1/\nu}\alpha M$ for finitely many  directions $\bs h$. We show by induction that these values are sufficient to recover each column of $\tilde M$. Once this is established, the requirement $\det M=1$ gives  $M=(\det \tilde M)^{-1/d}\tilde M$ and $\sigma^2\alpha^{2\nu}= (\det \tilde M)^{2\nu/d}$.

Let $\tilde M_{i,j}$ denote the $i,j^\text{th}$ element of $\tilde M$ and let $\tilde M_{:,i}$ denote the $i^\text{th}$ column of $\tilde M$. Also let $\tilde M_{1:k,1:k}$ be the submatrix with elements $\tilde M_{i,j}$ for $i,j=1,\ldots,k$.
%By our assumptions on $\tilde M$, $\tilde M_{i,i}>0$, $\prod_{i=1}^d \tilde M_{i,i}=\sigma ^{d/\nu} \alpha^d$ and $\tilde M_{i,j}=0$ when $i>j$. 
For the first column of $\tilde M$, notice that $|\tilde Me_1|=\tilde M_{1,1}$ where  $e_1,\ldots,e_d$ denote the standard basis of $\Bbb R^d$. This follows  since $\tilde M$ is upper triangular with positive diagonal.  For the inductive step suppose the first $k$ columns  $\tilde M_{:,1}\ldots,\tilde M_{:,k}$  are known. Taking $\bs h= e_{k+1}$  and $\bs h= e_{k+1}-e_{i}$   allows us to recover $ |\tilde M_{:,k+1}|^2$ and  $ |\tilde M_{:,k+1}-\tilde M_{:,i}|^2$ for $i=1,\ldots, k$. By adding and subtracting appropriate terms we can then recover:
$\bigl\langle \tilde M_{:,k+1} ,  \tilde  M_{:,i}  \bigr\rangle$, for all $i=1,\ldots,k+1$.
Therefore $\tilde  M_{:,k+1}= \Bigl(v, \sqrt{| \tilde M_{:,k+1}|^2-|v|^2 } ,0,\ldots,0\Bigr)^T$ where $v\triangleq \tilde M_{1:k,1:k}^{-1} \bigl(\bigl\langle  \tilde M_{:,k+1} ,   \tilde M_{:,i}  \bigr\rangle\bigr)_{i=1}^{k}$.
This establishes the inductive step and therefore $ \tilde M$ can be identified from observing $| \tilde M\bs h|^2$ at $d(d+1)/2$ different vectors $\bs h$ (let them be denoted by $\bs h_1,\ldots,\bs h_{d(d+1)/2}$).

Notice that as $\tilde M$ ranges over the set of upper triangular matrices with positive diagonal, the transformation $\{|\tilde M\bs h|\colon \bs h=\bs h_1\ldots,\bs h_{d(d+1)/2}\} \overset{f_1}\rightarrow \tilde M \overset{f_2}\rightarrow (M,\sigma^2\alpha^{2\nu})$ sends an open subset of $\Bbb R^{d(d+1)/2}$ to $ SL(d,\Bbb R)\times \Bbb R^+$. Since $f_2\circ f_1$ is a continuous map,
\[ (\widehat{\sigma^2\alpha^{2\nu}},\widehat{M})\rightarrow (\sigma^2\alpha^{2\nu},M), \quad w.p.1 \]
as $n\rightarrow \infty$. 
%\Comment{This is because $ \tilde M_{:,k+1}=f_{k+1}( \tilde M_{:,1},\ldots, \tilde M_{:,k}, x,y)$ where $x=| \tilde M_{:,k+1}|^2$ and $y=\bigl(\bigl\langle  \tilde M_{:,k+1} ,   \tilde M_{:,i}  \bigr\rangle\bigr)_{i=1}^{k}$,  The continuity of the transformation is important since it establishes the convergence }

%%%%%%%%%%%%%%%%%%%%%%%%%
\subsection{Estimating $\alpha$, when $d>4$}
%%%%%%%%%%%%%%%%%%%%%

In this section we construct an estimate of $\sigma^2 \alpha^{2\nu+2} |M\bs h|^{2\nu+2}$ when $d>4$, which, in combination with  $M$ and $\sigma^2 \alpha^{2\nu}$, allows us to consistently estimate $\alpha$.  We start by noticing that by  Lemma \ref{FirstLemma}, for any $p,q>\nu+1$
\[ \mathsf E n^2 \left[Q^p_n -\frac{a^p_\nu}{a^q_\nu}Q^q_n\right]\rightarrow \left[b^p_\nu-  \frac{a^p_\nu}{a^q_\nu} b^q_\nu\right] \]
as $n\rightarrow \infty$. 
The term $b^p_\nu-  \frac{a^p_\nu}{a^q_\nu} b^q_\nu$ is significant because, for any positive integer $p,q$
  \[b^p_\nu-  \frac{a^p_\nu}{a^q_\nu} b^q_\nu= c\,\sigma^2 \alpha^{2\nu+2} |M\bs h|^{2\nu+2} \]
   where $0\leq c\leq \infty$ is a known constant depending on $p$ and $q$. In addition, Lemma \ref{Pozz} in the Appendix establishes that  $c\neq 0$ and $c\neq \infty$ for at least one  $p,q>\nu+1$.
%The term $b^p_\nu-  \frac{a^p_\nu}{a^q_\nu} b^q_\nu$ is significant because, by Lemma \ref{Pozz} in the Appendix, there exists $p,q$ such that $b^p_\nu-  \frac{a^p_\nu}{a^q_\nu} b^q_\nu= c\,\sigma^2 \alpha^{2\nu+2} |M\bs h|^{2\nu+2} $ where $c$ is a known, non-zero constant. 
Moreover, $a^p_\nu/a^q_\nu$ doesn't depend on the unknown parameters $\sigma^2$, $\alpha$ and $M$ and therefore one can construct $n^2 \left[Q^p_n -\frac{a^p_\nu}{a^q_\nu}Q^q_n\right]$ from the observed values of the random field $Y$. The following theorem quantifies how large $p$ and $q$ need to be  for the almost sure convergence of $n^2 \left[Q^p_n -\frac{a^p_\nu}{a^q_\nu}Q^q_n\right]$ to $b^p_\nu-  \frac{a^p_\nu}{a^q_\nu} b^q_\nu$.

%------------------------
\begin{theorem}  
\label{secondThm}
Let $Y$ be a mean zero, geometric anisotropic $d$-dimensional Mat\'ern Gaussian random field with parameters $(\sigma, \alpha, \nu, M)$ and let $\Omega$ be a bounded, open subset of $\Bbb R^d$. Suppose $p\neq q$ are positive integers such that $p,q >\nu+1$ and both are large enough so that $4< \min\{2p-2\nu, d\}$ and  $4< \min\{2q-2\nu,d\}$.
Then
\[ n^2\left[Q^p_n -\frac{a^p_\nu}{a^q_\nu}Q^q_n\right]\rightarrow \left[b^p_\nu-  \frac{a^p_\nu}{a^q_\nu} b^q_\nu\right]  , \quad w.p.1\]
as $n\rightarrow \infty$.
\end{theorem}

Theorems \ref{FristThem} and \ref{secondThm} show that there exists strongly consistent estimates of ${\sigma^2 \alpha^{2\nu}}$, ${M}$ and ${\sigma^2 \alpha^{2\nu+2} |M\bs h|^{2\nu+2}}$.  This, in turn, gives consistent estimates of $\alpha$, $\sigma$ and $M$. Notice that when $d\leq 3$ this is impossible due to the mutual absolute continuity of Mat\'ern Gaussian random fields with different scale and variance parameters (see Zhang \cite{zhang:2004}). %For a detailed discussion of mutual absolute continuity and orthogonality for Gaussian random fields see \cite{stein:book},  \cite{wahba1990}, or \cite{Ibragimov}.  
Since Gaussian measures are either mutually absolutely continuous or orthogonal, the fact that we have strongly consistent estimates of $\alpha$, $\sigma$ and $M$ gives the following corollary.

%----------------------
\begin{corollary} Let $Y_1$ and $Y_2$ be two, mean zero, geometric anisotropic $d$-dimensional Mat\'ern Gaussian random fields defined a bounded open set $\Omega\subset \Bbb R^d$ with parameters $(\sigma_1,\alpha_1,\nu,M_1)$ and $(\sigma_2,\alpha_2,\nu,M_2)$ where $d>4$. If $(\sigma_1,\alpha_1)\neq (\sigma_2,\alpha_2)$ or $M_1\neq_{SL/SO}M_2$ then the Gaussian measures induced by the random fields ${Y_1}$ and $ {Y_2}$ are orthogonal. \end{corollary}

 {\it Remark:} The strong consistency results for our estimates of $\sigma^2 \alpha^{2\nu}$, $\alpha$ and $M$ all depend on knowledge of the true value of $\nu$. However, our results can be extended when using an estimate $\hat\nu$ so long as the error $\epsilon_n\triangleq\hat \nu-\nu$ satisfies  $\epsilon_n\log n\rightarrow 0$ with probability one as $n\rightarrow \infty$. This follows since the ratio of the quadratic variation, $Q^m_n$, using the true $\nu$, to the quadratic variation using the estimated $\hat \nu$, is $n^{-\epsilon_n}$ which converges to $1$ if  $\epsilon_n\log n\rightarrow 0$.

%%%%%%%%%%%%%%%%%%
\section{Beyond the Mat\'ern}
%%%%%%%%%%%%%%%
\label{BeyondB}
The previous section dealt exclusively with the Mat\'ern autocovariance. Now we show how these results can be extended to other autocovariance functions. We choose  two examples to illustrate how the methodology can be easily extended beyond the Mat\'ern autocovariance function. 
The key components for showing extensions are establishing versions of Lemmas  \ref{FirstLemma} and  \ref{boundonderiv}. Lemma \ref{FirstLemma} quantifies the expected value of the squared increments $(\Delta_{\bs h/n}^{p} Y(\bs t))^2$ in terms of $n$. Lemma \ref{boundonderiv} establishes that, in effect, derivatives of the covariance  away from the origin are dominated by the derivatives of the principle irregular term.
Once the analogs of these Lemmas are established  all the subsequent arguments for versions of Theorems \ref{FristThem} and \ref{secondThm} follow almost immediately. 

For our first example we consider the case when $Y$ is a mean zero Gaussian random field
on $\Bbb R^d$ with  generalized autocovariance function $c_1|t|^{\delta_1}+c_2 |t|^{\delta_2}$ where $\delta_1$ and $\delta_2$ are known but $c_1$ and $c_2$ are unknown  (it is tacitly assumed that the values of $c_1$ and $c_2$ give a conditionally positive definite function of order $\lfloor \delta_2/2 \rfloor$ in $\Bbb R^d$, see \cite{chiles:book}). In what follows we suppose $\delta_2>\delta_1>0$ and neither are even integers. The appropriate version of Lemma \ref{FirstLemma} says that when $p> \delta_2/2$  
 \begin{equation}
 \label{beyond1}
\mathsf E(\Delta_{\bs h/n}^{p} Y(\bs t))^2 = \frac{c_1 C_{p,\delta_1} }{n^{\delta_1}}+  \frac{c_2 C_{p,\delta_2}}{n^{\delta_2}}
\end{equation}
 where 
$C_{p,\delta} \triangleq  |\bs h|^{\delta}\sum_{i,j=0}^p (-1)^{i+j}{p\choose i}{p\choose j} |i-j|^{\delta} $.  Now $Q_n^p$ is defined as in (\ref{DefofQQ}) with $\delta_1$ in place of $2\nu$.
% \[
 %Q_n^p \triangleq  \frac{1}{\#\Omega_n} \sum_{\bs j\in\Omega_n} n^{\delta_1}(\Delta_{\bs h/n}^p Y(\bs j))^2.
%\]
In this case, $\mathsf E Q_n^p = c_1 C_{p,\delta_1}+ c_2 C_{p,\delta_2} n^{\delta_1-\delta_2}$ and therefore we set $\hat c_1\triangleq Q_n^p/C_{p,\delta_1} $. Also,  for an integer $q>p$ we have $\mathsf E n^{\delta_2-\delta_1}\bigl[  Q_n^p - \frac{C_{p,\delta_1}}{C_{q,\delta_1}} Q_n^q  \bigr]= c_2\bigl[  C_{p,\delta_1} - \frac{C_{p,\delta_1}}{C_{q,\delta_1}} C_{q,\delta_2}  \bigr]$ and after a renormalization one gets the estimate $\hat c_2$.
The analog to Lemma  \ref{boundonderiv}  says that when $p> \delta_2/2$ and $\Omega$ is a bounded  open subset of $\Bbb R^d$ there exists a constant $c>0$ such that 
\begin{equation}
 \label{beyond2}
 \bigl|\partial_{\bs h}^{(p,p)} \text{cov}(Y(\bs s),Y(\bs t)) \bigr|  \leq c |\bs s-\bs t|^{\delta_1-2p}
\end{equation}
for all $\bs s,\bs t\in \Omega$ such that $\bs s\neq \bs t$.  
Once (\ref{beyond1}) and (\ref{beyond2}) are established,  versions of Lemma \ref{bounds1}, Lemma \ref{lalalemma}, Lemma \ref{boundd} and Theorem \ref{FristThem}  following by replacing $2\nu$ with $\delta_1$. To establish  Theorem  \ref{secondThm}, replace the $n^2$ term with $n^{\delta_2-\delta_1}$ in equation (\ref{eeeps}) and continue in an similar manner to establish the following theorem. 
%-------------------
\begin{theorem}
Suppose $Y$ is a mean zero Gaussian random field on $\Bbb R^d$ with generalized autocovariance function $c_1|t|^{\delta_1}+c_2|t|^{\delta_2}$ observed on $\Omega\cap \{ \Bbb Z^d/n \}$ where $\Omega$ is a bounded open subset of $\Bbb R^d$ and $0<\delta_1<\delta_2$ are known and not even integers. 
If $0<2(\delta_2 - \delta_1 )< d$ then there exists integers $q>p>0$ such that $\hat c_1$ and $\hat c_2$ (defined above) converge with probability one to $c_1$ and $c_2$ (respectively) as $n\rightarrow\infty$.
%
%If $p>\delta_2/2$ then $\hat c_1\rightarrow c_1$ with probability one as $n\rightarrow \infty$. If, in addition,  
%$0<2(\delta_2 - \delta_1 )< d$  and $p>(\delta_1+d)/2$, then there exists an integer $q>p$ such that $\hat c_2\rightarrow c_2$ with probability one as $n\rightarrow \infty$.
\end{theorem}
There are different conditions on $p$  to guarantee convergence of $c_1$ versus  $c_2$. Generally, one only needs $p>\delta_1/2$ for consistent estimation of $c_1$, which will hold in any dimension. However, in our case, we need the additional requirement that $p>\delta_2/2$ since we are working with a conditionally positive definite function of order $\lfloor \delta_2/2\rfloor$.   To get consistent estimation of $c_2$ we need the additional inequality $2(\delta_2-\delta_1)< \min\{ 2p-\delta_1, d\}$. To relate this to our  Mat\'ern results in Section \ref{Geoo} set $\delta_1=2\nu$ and $\delta_2=2\nu+2$ so that the inequality becomes $4<\min\{ 2p-2\nu,d\}$ which appears in Theorem \ref{secondThm}.  Finally the analog to Lemma \ref{Pozz} guarantees there exits a $q>p$ such that $\bigl[  C_{p,\delta_1} - \frac{C_{p,\delta_1}}{C_{q,\delta_1}} C_{q,\delta_2}  \bigr]$ is non-zero which allows us to define $\hat c_2$.

Before we continue, we mention a comment in Wahba's book  (\cite{wahba1990}, page 44) which argues in favor of using the generalized autocovariance $|t|^{2m-1}$ over the model $|t|^{2m-1}+ c_1|t|^{2m+1}+\cdots +c_k|t|^{2m+2k-1}$ when $d=1,2,3$. The reasoning is that the two models yield mutually absolutely continuous Gaussian measures, and therefore can not be consistently distinguished.  We can  see, however, that the dimension requirement $d=1,2, 3$ is an integral component of this argument.  When the dimension gets above $4$, this reasoning no longer holds since the two models are orthogonal by the above theorem (setting $\delta_1=2m-1$ and $\delta_2=2m+1$).

For our second extension we show that the variance $\sigma^2$ and scale $\alpha$ can be separately estimated  in the exponential autocovariance model $\sigma^2 e^{- |\alpha t|^\delta}$ when the dimension  $d>2\delta$ and $\delta\neq 1$. In this case, the appropriate version of Lemma   \ref{FirstLemma}  becomes
 \begin{equation}
 \label{beyond3}
\mathsf E(\Delta_{\bs h/n}^{p} Y(\bs t))^2 = -\frac{\sigma^2 \alpha^\delta C_{p,\delta} }{n^{\delta}}+  \frac{\sigma^2\alpha^{2\delta} C_{p,2 \delta}}{2n^{2\delta}} +O(n^{-3\delta})
\end{equation}
as $n\rightarrow \infty$ when $p>\delta/2$. 
From (\ref{beyond3}) one can now easily construct estimates of $\sigma^2 \alpha^\delta$ and $\sigma^2\alpha^{2\delta}$. When a geometric anisotropy $M$ is present, the techniques of Section \ref{Geoo} are also sufficient to also construct $\widehat{M}$.
Notice that by direct differentiation, equation (\ref{beyond2}) holds when $\delta_1$ is replaced by $\delta$.  
Using similar arguments for the previous theorem and extending to a geometric anisotropy the following theorem is obtained. 

%--------------------
\begin{theorem}
Let $Y$ be a mean zero, Gaussian process on $\Bbb R^d$ with autocovariance function  $\sigma^2 e^{-|\alpha M \bs t|^\delta}$  observed on  $\Omega\cap \{ \Bbb Z^d/n \}$ where $\Omega$ is a bounded open subset of $\Bbb R^d$. Suppose $\delta\in(0,2)$ is known,  $\sigma$ and $\alpha $ are positive and $M$ is upper triangular with positive diagonal and determinant 1. If  $p\geq1$ then $\widehat{\sigma^2 \alpha^\delta}\rightarrow \sigma^2 \alpha^\delta$ and $\widehat{M}\rightarrow M$ with probability one as $n\rightarrow \infty$. Moreover, if $2\delta<d $ and $\delta\neq 1$ then for any $p>3\delta/2$ there exists $q>p$ such that $\hat \sigma\rightarrow \sigma$ and $\hat \alpha\rightarrow \alpha$ with probability one as $n\rightarrow \infty$.
\end{theorem}
%In particular, when $\delta<1/2$ one can consistently estimate both the scale $\alpha$ and $\sigma$ in any dimension. Moreover, if $d>4$ this holds for all $\delta\in(0,2)$. \Comment{You might want to drop the stuff with the min term in the above theorem and let the reader figure it out from the comments above.}

Many other extensions are possible, including more general non-stationary random fields.
In this case, both $a_\nu^m$ and $b_\nu^m$ depend on $\bs t\in \Omega$ and $Q_n^p$  will convergence to $\int_\Omega a_\nu^m d\bs t$ and  similarly for $\int_\Omega  \left[b^p_\nu-  \frac{a^p_\nu}{a^q_\nu} b^q_\nu\right] d\bs t$. If one also needs pointwise convergence to $a_\nu^m$ or $b^p_\nu-  \frac{a^p_\nu}{a^q_\nu} b^q_\nu$  one can consider weighted local averaging of the terms in $Q_n^p$. This was the technique used in  \cite{anderesSourav} when observing a deformed isotropic Gaussian random field that locally behaved like a fractional Brownian field. However, obtaining extensions in these cases are more difficult since one needs to consider rates of decay for a bandwidth parameter. That being said, this work  leaves open the possibility of constructing consistent estimates of the  two deformations $f_1,f_2$ when observing $Y_1\circ f_1+ Y_2\circ f_2$ where $Y_1$ and $Y_2$ have generalized autocovariance functions $|t|^{\delta_1}$ and $|t|^{\delta_2}$ respectively.
Finally we mention that since $Q_n^p$ is constructed from increments, one can extend our results to random fields $Y$ with a polynomial drift of known order. 
%
%Our methods also establish that the  Mat\'ern field is orthogonal to a field with generalized autocovariance  which has the same principle irregular term when $d>4$.
%\begin{theorem}
%Let $Y_1$, $Y_2$ be, mean zero Gaussian processes on $\Bbb R^d$. Suppose $Y_1$ has an isotropic Matern autocovariance given by equation  (\ref{DefineK}) with parameters $\sigma,\alpha,\nu>0$. Suppose $Y_2$ has a generalized autocovariance $\sigma^2 G_{\nu}(\alpha t)$.
%If $d>4$ then the Gaussian measures induced by the random fields ${Y_1}$ and $Y_2$ are orthogonal.
%\end{theorem}
% In contrast, Stein shows in  \cite{stein:equiv} that these two fields are equivalent when $d\leq 3$. The heuristic here is when $d>4$ there is enough spatial information to detect the precence of a 'second irregular term' which occurs in the Matern and not in the generalized autocovariance $G_\nu$.

%

%

%%%%%%%%%%%%%%%%%%%%%%
\section{Simulations}
%%%%%%%%%%%%%%%%%%%%%%%
 We finish with two simulations that illustrate (and hopefully compliment) our theoretical results. The first simulation shows how one can use directional increments to estimate $\sigma^2 \alpha^{2\nu}$ and a geometric anisotropy $M$ using finitely many directions. The second simulation shows how to estimate the coefficient on the \lq second principle irregular term' ($c_2$ in equation (\ref{PrinceTerm})) and how it can be used to construct an unbiased estimate of the coefficient on the \lq first principle irregular term' ($c_1$ in equation (\ref{PrinceTerm})).

In our first example, we simulated 500 independent realizations of a Mat\'ern random field with parameters $\sigma=1.5$, $\alpha=0.8$, $\nu= 1.75$, $M(1,1)=1.2$, $M(1,2)=0.5$, $M(2,1)=0$ and $M(2,2)=1/1.2$ observed on a square grid in $[0,1]^2$ with spacing  $1/55$. On each realization we
estimated $\sigma^2 \alpha^{2\nu}$ and $M$ using $2$, $3$ and $4$ horizontal, vertical and diagonal increments.
%$\sigma=1.5$, $\alpha=0.8$, $\nu=1.75$ and $M=\begin{pmatrix} 1.2 & 0.5\\0 &1/1.2\end{pmatrix}$. 
Notice that since $1<\nu<2$, this random field is once, but not twice, mean square differentiable.  Intuitively, we therefore need at least two increments for sufficient de-correlation of the terms in the quadratic variation sum (\ref{PrinceTerm}).  Table \ref{tbl1} displays the root mean squared error (RMSE) for estimating  $\sigma^2 \alpha^{2\nu}$, the true value is approximately $1.03$, and the elements of $M$.  Figure \ref{fig2} plots histograms of the estimates for  $2$ and $3$ increments. It is immediately clear that there is a large reduction in RMSE when using $3$ increments as compared to  $2$ increments (and an additional bias reduction when estimating $\sigma^2\alpha^{2\nu}$). Indeed, by Theorem \ref{FristThem}, more increments leads to more spatial decorrelation and hence a reduction in variance. In this case, $\nu<2<\nu+1$ so that the estimate based on $2$ increments is guaranteed to be consistent but the variance decays at a sub-optimal rate. Since $3>(4\nu+d)/4=2.25$, the variance of the estimate based on $3$ increments decays at the optimal rate. However, Theorem \ref{FristThem} also says that this variance reduction only holds up to a point, after which taking more increments no longer effects the rate of variance decay.   Indeed, it is seen in Table \ref{tbl1} that taking 4 increments do not improve the RMSE nearly as much. 
\begin{table}[h]
  \caption{\label{tbl1} RMSE for estimating $\sigma^2 \alpha^{2\nu}$ and $M$ using $2$, $3$ and $4$ increments.}
 \begin{tabular}{l|ccc}
 & 2 increments & 3 increments & 4 increments \\
 \hline 
$\sigma^2\alpha^{2\nu}$ &  0.1664   & 0.0300   & 0.0289 \\
$M(1,1)$ &  0.0360  &  0.0114  &  0.0113 \\
$M(1,2)$ & 0.0475  &  0.0147  &  0.0147 \\
$M(2,2)$ & 0.0248  &  0.0079  &  0.0079
 \end{tabular}
\end{table}

Our second simulation uses the results of Section \ref{BeyondB} to estimate $c_1$ and $c_2$ when observing   $\sqrt{c_1}\,Y_1 + \sqrt{c_2}\, Y_2$  on $[0,1/\sqrt{2})^2$ at $1000\times 1000$ pixel locations where $c_1=100$, $c_2=36$ and $Y_1$ is independent of $Y_2$.  The random field $Y_1$ has autocovariance  $\frac{9}{10}-|t|^{0.2}+\frac{1}{10}|t|^2$ and  $Y_2$ has autocovariance $\frac{8}{10}-|t|^{0.4}+\frac{2}{10}|t|^2$ which is positive definite on $[0,1/\sqrt{2})^2$ (see \cite{stein:FastSim} for a proof). 
Our estimates of $c_1$ and $c_2$ are defined by 
\begin{align}
\hat c_1&\triangleq Q_n^p/C_{p,\delta_1}\\
\hat c_2&\triangleq n^{\delta_2-\delta_1}\frac{  Q_n^p - \frac{C_{p,\delta_1}}{C_{q,\delta_1}} Q_n^q }{ C_{p,\delta_2} - \frac{C_{p,\delta_1}}{C_{q,\delta_1}} C_{q,\delta_2} }
\end{align}
where $\delta_1=0.2$, $\delta_2=0.4$, $p=2$, $q=3$ and  $C_{p,\delta} \triangleq - |\bs h|^{\delta}\sum_{i,j=0}^p (-1)^{i+j}{p\choose i}{p\choose j} |i-j|^{\delta} $.
 This example was chosen to illustrate the duality when estimating $c_1$ and $c_2$: the smaller $|\delta_1-\delta_2|$  (in relation to the dimension $d$)  the smaller the variance of $\hat c_1$ and $\hat c_2$ but the larger the bias of $\hat c_1$. 
In fact, as the dimension grows, the variance  $\hat c_1$ decreases at a faster rate (proportional to $n^{-d}$ when using enough increments) but the bias decreases at the same asymptotic rate for any $d$ (proportional to $n^{\delta_1-\delta_2}$). 
In our example, since $p=2$ (so the quadratic term $\frac{1}{10}|t|^2$ vanishes), we can explicitly compute the bias using equation (\ref{beyond1}) so that $ \mathsf E \hat c_1=c_1+ c_2 \frac{C_{p,\delta_2}}{C_{p,\delta_1}} n^{\delta_1-\delta_2}$.
Notice that using our estimate of $c_2$  we can now correct the bias in $\hat c_1$.
%For this simulation we used increments so that the estiamte of $c_2$ is unbias and bias of $\hat c_1$ is $c_2 \frac{C_{p,\delta_2}}{C_{p,\delta_1}} n^{-0.2} $ and we get the optimal variance rate for $c_1$. 
 %Notice, however, that the bias of $\hat c_1$, in this case is $c_2 \frac{C_{p,\delta_2}}{C_{p,\delta_1}} n^{-0.2} $ which can be estimated by $\hat c_2\frac{C_{p,\delta_2}}{C_{p,\delta_1}} n^{-0.2}$. 
 The left plot of Figure \ref{fig3} shows two histograms of the estimate $\hat c_1$ and the bias corrected estimate  $\hat c_1- \hat c_2\frac{C_{p,\delta_2}}{C_{p,\delta_1}} n^{\delta_1-\delta_2}$ on the 500 simulated realizations. The right plot of Figure \ref{fig3} shows the histogram of the estimate $\hat c_2$. We can see that not only is it possible to get an estimate of $c_2$, but using it to correct the bias in  $\hat c_1$ reduces the RMSE for estimating $c_1$ (from $7.84$ down to $2.29$).

%\begin{center}
% \begin{tabular}{r|l}
% & RMSE \\
% \hline 
%$\hat c_1$ &  7.8358\\
% Bias corrected $\hat c_1$ & 2.2879 \\
%$\hat c_2$ &  9.9845
% \end{tabular}
%\end{center}

%%We finish by mentioning  that when $d=1,2,3$ and $4(\nu-m)<-d$ 
%%$(Q_n^m-a_\nu^m)/\sqrt{\text{var}(Q_n^m)}$ converge in distribution to $\mathcal N(0,1)$.  This follows by a standard Lindberg-Feller argument. Also by Theorem \ref{FristThem}  $\text{var} Q_n^m =O(n^{-d})$ so that estimates of $\sigma^2 \alpha^{2\nu}$ are "root n" consistent. Combing this estimate with a one-newton-step maximum likelihood leaves open the possibility of constructing estimates of  $\sigma^2 \alpha^{2\nu}$ which is asymptotically efficient.  
%% Ying gets an asymptotic variance of $2\sigma^4\alpha^{2\nu}$ for the exponential case when $\nu=1/2$ and $d=1$. If you use $m=1$, then $2(a_\nu^m)^2=2(2 \sigma^2\alpha^{2\nu})^2=8\sigma^4\alpha^{2\nu}$.

\begin{figure}[h]
\centering
{\includegraphics[height=3.8in]{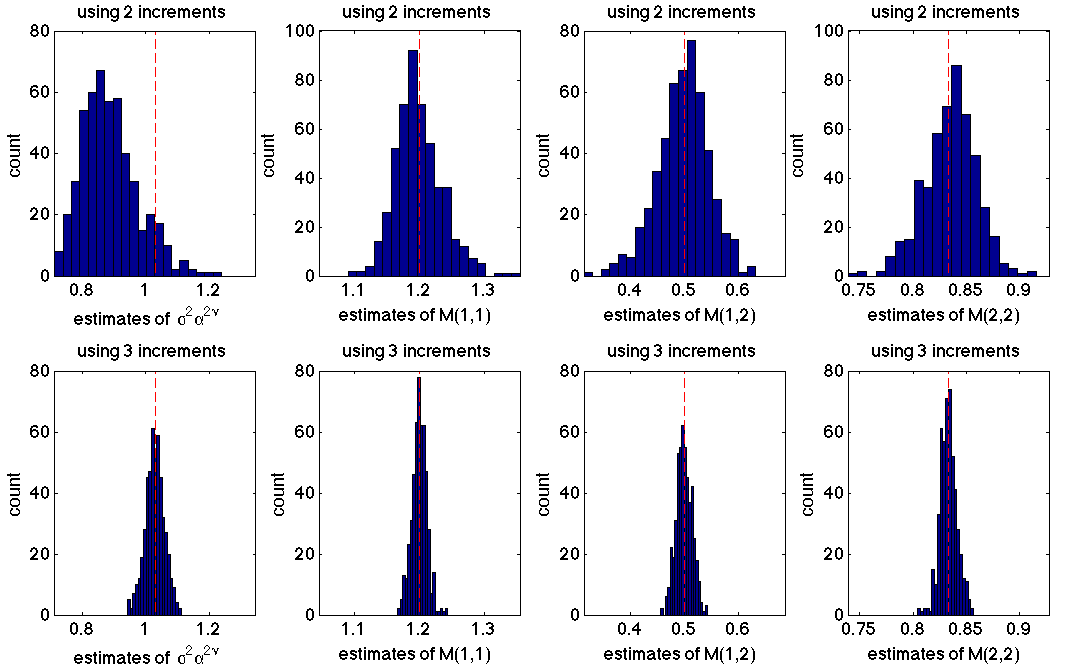}}
\caption{500 independent simulations of a Mat\'ern random field with $\sigma=1.5$, $\alpha=0.8$, $\nu= 1.75$, $M(1,1)=1.2$, $M(1,2)=0.5$, $M(2,1)=0$ and $M(2,2)=1/1.2$ observed on a square grid in $[0,1]^2$ with spacing   $1/55$.
The top row of figures shows the histograms of the estimates of $(\sigma^2 \alpha^{2\nu}, M(1,1), M(1,2), M(2,2))$ using the techniques derived in Section \ref{AnyDim} based on increments of order 2. The bottom row shows the histograms of the estimates using increments of order 3.}
\label{fig2}
\end{figure}

\begin{figure}[h]
\centering
{\includegraphics[height=2.2in]{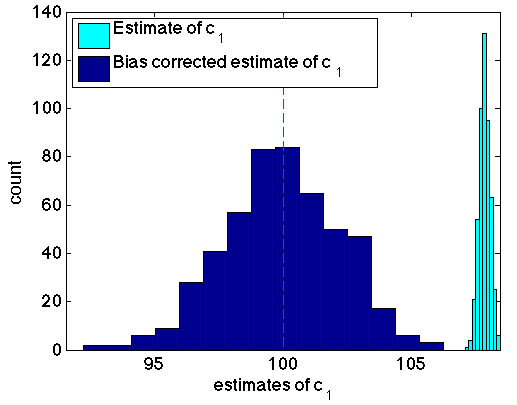}}
{\includegraphics[height=2.2in]{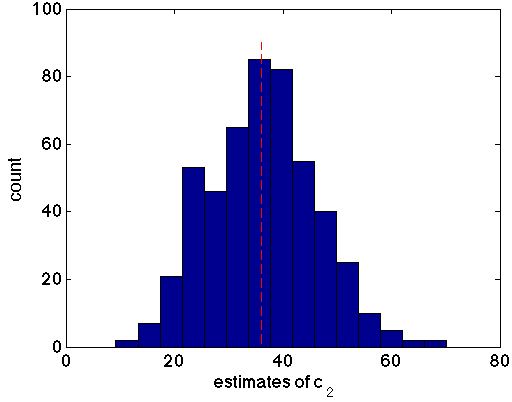}}
%{\includegraphics[height=1.8in]{Fig3c.png}}
\caption{Histograms of the estimates of  $c_1$ and $c_2$ for 500 independent realizations of $\sqrt{c_1}\,Y_1 + \sqrt{c_2}\, Y_2$ where $c_1=100$, $c_2=36$ and $Y_1$ is independent of $Y_2$. The random field $Y_1$ has principle irregular term $-|t|^{0.2}$ and  $Y_2$ has principle irregular term $-|t|^{0.4}$. Each realization is on $[0,1/\sqrt{2})^2$ measured at $1000\times 1000$ pixel locations. \label{fig3}}
\end{figure}

%%%%%%%%%%%%%%%%%%%%%%%%%%%%
\appendix
%%%%%%%%%%%%%%%%%%%%%%%%%%%%%%

\section{Proofs}
We start with some notation.
For a function of two variables $F(\bs s,\bs t)$ let $\Delta ^{(m,n)}_{\bs h} F(\bs s,\bs t)\triangleq\Delta^m_{\bs h}\Delta^n_{\bs h}F(\bs s,\bs t)$ where $\Delta^m_{\bs h}$ acts on the variable $\bs s$ and $\Delta^n_{\bs h}$ acts on the variable $\bs t$. Define $\partial_{\bs h}\triangleq\bs h\cdot \nabla$ to be the  directional derivative in the direction $\bs h$ and $\partial_{\bs h}^{(m,n)} F(\bs s,\bs t)\triangleq \partial^m_{\bs h} \partial^n_{\bs h}F(\bs s,\bs t)$  where $\partial^m_{\bs h}$  acts on the variable $\bs s$ and $\partial^n_{\bs h}$ acts on  $\bs t$.

Let $f(\xi), g(\xi)$ be real valued functions defined on some set $\Xi$ and let $\Xi^\prime\subset\Xi$. We write $f(\xi)\lesssim g(\xi)$ for all $\xi\in \Xi^\prime$ if there there exists a positive constant $c>0$ such that  $|f( \xi)|\leq c\, g( \xi)$ for all $ \xi\in\Xi^\prime$.  Notice that this definition also works for a sequence of functions $f_n, g_n$  by considering the variable $n$ as an argument and replacing $\Xi$ by $\Xi\times \Bbb N$. 
%In particlar, when  $f_n,g_n$ are real valued spatial functions for each $n\in \Bbb N$,  $f_n(\bs s)\lesssim g_n(\bs s)$  for all $n>N$ and $\boldsymbol s\in\Theta$ means that form some $c>0$, $|f_n(\boldsymbol s)|\leq c\, g_n(\boldsymbol s)$ for all  $n>N$ and $\boldsymbol s\in\Theta$.  

%------------------------------------

\begin{proof}[\rm \bf Proof of Lemma \ref{FirstLemma}]
 We suppose $\sigma=\alpha=1$ and $M$ is the identity matrix, then rescale for the general case. First note two immediate facts about the $m^\text{th}$ directional increment operator $\Delta_{\bs h/n}^{m}$:   for any function $f\colon \Bbb R^d\rightarrow \Bbb R$ the $m^\text{th}$-increment of $f$ can be computed $\Delta_{\bs h/n}^m f(\bs t)=\sum_{i=0}^m d_i f(\bs t+ i\bs h/n)$ where $d_i=(-1)^{m+i}{m \choose i}$;  The $m^\text{th}$-increment $\Delta_{\bs h/n}^{m}$ annihilates  monomials of degree less than $m$ so that $ \Delta_{\bs h/n}^{(m,m)}|\bs t-\bs s|^{2k}=0$ for all $k=0,\ldots,m-1$. Therefore, by the expansions given on page 375 of  \cite{AShandbook} 
 we have
\[ \Delta_{\bs h/n}^{(m,m)} K(|\bs s-\bs t|)=  \Delta_{\bs h/n}^{(m,m)} \Bigl\{ G_\nu(|\bs s-\bs t|)-\nu G_{\nu+1}(|\bs s-\bs t|)+r(|\bs s-\bs t|) \Bigr\}\]
where $ r(\epsilon)=o(\epsilon^{2\nu+2})$ as $\epsilon\rightarrow 0$. 
Now for a fixed $\bs t_0\in \Bbb R^d$
\[ \mathsf E(\Delta_{\bs h/n}^m Y(\bs t_0))^2= \Delta_{\bs h/n}^{(m,m)} \Bigl\{ K(|\bs s-\bs t|) \Bigr\}\Bigr|_{\bs s,\bs t=\bs t_0}=\mathcal I_1 + \mathcal I_2+\mathcal I_3 \]
where
\begin{align}
\mathcal I_1 & \triangleq   \Delta_{\bs h/n}^{(m,m)} \Bigl\{ G_\nu(|\bs s-\bs t|)\Bigr\}\Bigr|_{\bs s,\bs t=\bs t_0}= \sum_{ij}d_i d_j G_\nu(| (i-j)\bs h/n |)\\
\mathcal I_2 & \triangleq   \Delta_{\bs h/n}^{(m,m)} \Bigl\{ (-\nu)G_{\nu+1}(|\bs s-\bs t|)\Bigr\}\Bigr|_{\bs s,\bs t=\bs t_0}= \sum_{ij}d_i d_j (-\nu)G_{\nu+1}(| (i-j)\bs h/n |)\\
\mathcal I_3 & \triangleq \Delta_{\bs h/n}^{(m,m)} \Bigl\{r(|\bs s-\bs t|))\Bigr\}\Bigr|_{\bs s,\bs t=\bs t_0}= \sum_{ij}d_i d_j r(| (i-j)\bs h/n | )
\end{align}
Notice that $\sum_{ij}d_i d_j G_\nu(| (i-j)\bs h/n |)=|\bs h/n|^{2\nu} \sum_{ij}d_i d_j G_\nu(| i-j |)$. This is obviously true with $\nu\not\in \Bbb Z$. It also holds when $\nu\in\Bbb Z$ since
\begin{equation}
\label{rescaleMe}
G_\nu(| (i-j)\bs h/n |)= | \bs h/n|^{2\nu}\bigl(G_\nu(|i-j|)+|i-j|^{2\nu}\log |\bs h/n| \bigr)
\end{equation}
and  $ \sum_{ij}d_i d_j   |i-j|^{2\nu}=0$ (since $\nu\in\Bbb Z$ and $m>\nu$). Similar arguments can be applied  to $G_{\nu+1}$ when $m>\nu+1$ which gives
$ \mathcal I_1+\mathcal I_2= \frac{a_\nu^m}{n^{2\nu}}+\frac{b_\nu^m}{n^{2\nu+2}}$.
Finally, notice that $r(\epsilon)=o(\epsilon^{2\nu+2})$ implies that $\mathcal I_3=o(n^{-2\nu-2})$. This establishes the claim when $\sigma=\alpha=1$ and $M$ is the identity matrix.
The general result when $\sigma,\nu>0$ and $M\in GL(d,\Bbb R)$ is then established by an easy rescaling argument (using equation (\ref{rescaleMe}) when $\nu\in \Bbb Z$).
\end{proof}

%---------------------------------
\begin{lemma} 
\label{Pozz}
For $\nu>0$, let $a^m_\nu$ be defined by (\ref{AAA}) and $b^m_\nu$ be defined by  (\ref{BBB}). If $m>\nu$ then $a^m_\nu\neq 0$. If $m>\nu+1$ then $b^m_\nu\neq 0$. Finally, there exits $p,q>\nu+1$ such that  $b^p_\nu-  \frac{a^p_\nu}{a^q_\nu} b^q_\nu\neq 0$.
\end{lemma}
\begin{proof}
Notice first that $a_\nu^m\propto\text{var}(\Delta_1^m Z_\nu)>0$ where $Z_\nu$ is an intrinsic random function on $\Bbb R$ observed on $\Bbb Z$ with generalized covariance $G_\nu$ (since $\Delta_1^m$ annihilates polynomials of order $m-1$, and $m>\nu$, see \cite{stein:book}). The same reasoning establishes that $-b_\nu^m\propto \text{var}(\Delta_1^m Z_{\nu+1})>0$ when $m>\nu+1$.

For the last part of the lemma we show that there exists $p,q>\nu+1$ such that  
\[
\frac{\text{var}(\Delta_1^{p} Z_\nu)}{\text{var}(\Delta_1^qZ_\nu)} \neq \frac{\text{var}(\Delta_1^pZ_{\nu+1})}{\text{var}(\Delta_1^qZ_{\nu+1})}.
\]
We will argue by contradiction and suppose that for all $k>0$,
\begin{equation}
\label{contradictMe}  
\frac{\text{var}(\Delta_1^{q+k} Z_\nu)}{\text{var}(\Delta_1^qZ_\nu)} = \frac{\text{var}(\Delta_1^{q+k}Z_{\nu+1})}{\text{var}(\Delta_1^qZ_{\nu+1})}.
\end{equation}
By a spectral representation of $G_\nu$ (see \cite{stein:book} page 36) and an easy induction establishes that $\text{var}(\Delta_1^{q+k} Z_\nu)=\int |e^{iw}-1|^{2q+2k} |w|^{-2\nu-1}\,dw $ and $\text{var}(\Delta_1^{q+k} Z_{\nu+1})=\int |e^{iw}-1|^{2q+2k} |w|^{-2\nu-3}\,dw $. 
Notice also that $|e^{iw}-1|^2=2-2\cos\,w$.
Let $F_{\nu}$ and $F_{\nu+1}$ be two probability measures on $\Bbb R$ defined by 
\begin{align*}
F_\nu(B)&\triangleq\frac{1}{\text{var}(\Delta_1^q Z_\nu)} \int_B  (2-2\cos \, w)^{q} |w|^{-2\nu-1}\,dw \\
 F_{\nu+1}(B)&\triangleq\frac{1}{\text{var}(\Delta_1^q Z_{\nu+1})}\int_B  (2-2\cos \, w)^{q} |w|^{-2\nu-3}\,dw.
 \end{align*}
Our assumption (\ref{contradictMe}) then becomes
\begin{equation}
\label{contradictMe2}
 \int (2-2\cos \, w)^{k} d F_{\nu}(w)= \int (2-2\cos \, w)^{k} d F_{\nu+1}(w), 
 \end{equation}
for all $k>0$. 
 Notice that the variances $\text{var}(\Delta_1^q Z_\nu)$ and $\text{var}(\Delta_1^q Z_{\nu+1})$ serve as the normalizing constants so that $F_\nu$ and $F_{\nu+1}$ have total mass one. In what follows we show that the normalizing constants satisfy both $\text{var}(\Delta_1^q Z_\nu)>\text{var}(\Delta_1^q Z_{\nu+1})$ and $\text{var}(\Delta_1^q Z_\nu)<\text{var}(\Delta_1^q Z_{\nu+1})$ to establish the desired contradiction.

By the equalities in (\ref{contradictMe2}), the random variables $2(1-\cos W_\nu)$ and $2(1-\cos W_{\nu+1})$ have the same moments when $W_{\nu}\sim F_{\nu}$ and $W_{\nu+1}\sim F_{\nu+1}$. In addition, $0\leq 2(1-\cos W_\nu) \leq 4$ and $0\leq 2(1-\cos W_{\nu+1}) \leq 4$ so that the moment generating functions are both finite in a non-empty radius of the origin. Therefore $
 2(1-\cos W_\nu)\overset{\mathcal L}= 2(1-\cos W_{\nu+1})$,
 where $\overset{\mathcal L}=$ denotes equality in law.
This gives $\mathsf P(\cos W_\nu<0)=\mathsf P(\cos W_{\nu+1}<0)$, for example. However
\begin{align*}
 \mathsf P(\cos W_\nu<0)&= \frac{1}{\text{var}(\Delta_1^qZ_\nu)} \int  \bs 1_{\{\cos w<0\}}(2-2\cos w )^q |w|^{-2\nu-1}dw \\
 &> \frac{1}{\text{var}(\Delta_1^qZ_\nu)} \int  \bs 1_{\{\cos w<0\}}(2-2\cos w )^q |w|^{-2\nu-3}dw,
 \end{align*}
by the fact that $\cos w<0\Rightarrow |w|>\pi/2$. Therefore
\begin{equation}
\label{halfway}
\text{var}(\Delta_1^q Z_{\nu+1})< \text{var}(\Delta_1^q Z_\nu).
\end{equation}

To show the contradicting inequality let's start by computing the density of these two random variables. The idea is to  show that the non-normalized (i.e.\! without the term $\text{var}(\Delta_1^q Z_\nu)$) density of $2(1-\cos W_\nu)$ is strictly smaller than the non-normalized density of $2(1-\cos W_{\nu+1})$ in a positive neighborhood of $0$. In particular, the density of $2(1-\cos W_\nu)$
 can be written as $2\sum_{k=1}^\infty f_{W_\nu}(g_k(x)) |g_k(x)^\prime |$
 where the $g_k$'s are the different positive branches of the inverse $\cos^{-1}(1-x/2)$ and $f_{W_{\nu}}(w)\triangleq (2-2\cos w)^q |w|^{-2\nu-1}/\text{var}(\Delta_1^q Z_\nu)$ is the density of $W_\nu$. This simplifies to
\[ \frac{2x^q}{\text{var}(\Delta_1^q Z_\nu) }\sum_{k=1}^\infty  \frac{ |g_k(x)^\prime |}{|g_k(x)|^{2\nu+1}}= \frac{2x^q}{\text{var}(\Delta^q Z_\nu)\sqrt{x-x^2/4} }\sum_{k=1}^\infty  {|g_k(x)|^{-2\nu-1}}\]
for $0<x<4$. Notice that $g_1(x)\sim \sqrt{x}$ as $x\rightarrow 0$ and $g_k(x)\sim 2\pi \lfloor k/2\rfloor$ as $x\rightarrow 0$ for all $k>1$.  Therefore the term $g_1$ dominates the sum when $x$ is small. In particular for all $x>0$ sufficiently small 
we have
\begin{align}
%\frac{2x^q}{\text{var}(\Delta^q Z_\nu)\sqrt{x-x^2/4} }\sum_{k=1}^\infty  {|g_k(x)|^{-2\nu-1}} 
f_{2-2\cos W_\nu}(x)&< \frac{2x^q}{\text{var}(\Delta_1^q Z_\nu)\sqrt{x-x^2/4} }\sum_{k=1}^\infty  {|g_k(x)|^{-2\nu-3}} \\
&=\frac{\text{var}(\Delta_1^q Z_{\nu+1})}{\text{var}(\Delta_1^q Z_\nu)} f_{2-2\cos W_{\nu+1}}(x).
\end{align}
Since $f_{2-2\cos W_\nu}(x)$ and $f_{2-2\cos W_{\nu+1}}(x)$ have the same integrate integrals over Borel subsets of $(0,4)$, we must have
$ \text{var}(\Delta_1^q Z_{\nu+1})> \text{var}(\Delta_1^q Z_\nu)$.
 This contradicts (\ref{halfway}) and therefore establishes the lemma.
\end{proof}

%-------------------------------------
\begin{lemma} 
\label{ksqrt}
For any $\nu>0$, $T>0$,
 \[ \Bigl| \frac{d^p}{dt^p} t^{\nu/2} \mathcal K_{\nu}(\sqrt{t})\Bigr| \lesssim \begin{cases}
  1,& \text{when $p< \nu$;} \\
  |\log t|, &   \text{when $p= \nu$;} \\
   t^{\nu-p}, & \text{when $p> \nu$;} 
    \end{cases}
  \]
 as $t$ ranges in the interval $(0,T)$ where $\mathcal K_\nu$ is the modified Bessel function of the second kind of order $\nu$. \end{lemma}

\begin{proof}
Using the expansions for $\mathcal K_\nu$ found in \cite{AShandbook} (page 375) we can write
\begin{align}
\label{eqnBessel}
 t^{\nu/2}\mathcal K_\nu(\sqrt{t}) &= \begin{cases}
 F_1(t)  + t^{\nu}\log(t) F_3(t); &\text{when $\nu=0,1,2,\ldots$}\\
 F_4(t) + t^{\nu} F_5(t); & \text{otherwise}
 \end{cases}
 \end{align}
where the $F_j(t)$'s are of the form $\sum_{k=0}^\infty c_k t^{k}$ where the $c_k$'s decay fast enough so that the series converges absolutely for  all $t\in(0, \infty)$ and all it's derivatives exist and are bounded on $(0,T)$. This immediately establishes that when $p< \nu$, $ \bigl| \frac{d^p}{dt^p} t^{\nu/2} \mathcal K_{\nu}(\sqrt{t})\bigr| \lesssim 1$ for all $t\in(0,T)$ since 
both $\frac{d^p}{dt^p} (t^{\nu})$ and $\frac{d^p}{dt^p} (t^{\nu} \log t)$ are continuous and bounded on $(0,T)$. 
%
%Therefore by (\ref{eqnBessel}) we get $ \bigl| \frac{d^p}{dt^p} t^{\nu/2} \mathcal K_{\nu}(\sqrt{t})\bigr| \lesssim 1$ as $t$ ranges in the interval $(0,T)$.

When $p> \nu$ and $\nu\not\in\Bbb Z$ we have that 
$t^{\nu}\lesssim \frac{d}{dt} (t^{\nu}) \lesssim \cdots  \lesssim \frac{d^p}{dt^p} (t^{\nu})  \lesssim  t^{\nu-p}$
as $t$ ranges in the bounded interval $(0,T)$. Similarly, when $p> \nu$ and $\nu\in\Bbb Z$ we have
\[ t^{\nu}\log t\lesssim \frac{d}{dt} (t^{\nu} \log t) \lesssim \cdots  \lesssim \frac{d^p}{dt^p} (t^{\nu} \log t ) \lesssim t^{\nu-p}.  \]
%This follows  since $\frac{d^p}{dt^p} (t^{\nu} \log t)\propto  t^{\nu-p} \log t + c_p t^{\nu-p} $ for constants $c_p$ when $p<\nu\in \Bbb Z$. 
 Finally, when $p=\nu$, $\frac{d^p}{dt^p} t^{\nu} \log t\propto \log t+c_p$.
The lemma now follows by equation (\ref{eqnBessel}) and the fact that the derivative of a product satisfies $(fg)^{(p)}=\sum_{k=0}^p {p \choose k} f^{(p)} g^{(p-k)}$.

\end{proof}

%-----------------------------------------
\begin{lemma} \label{boundonderiv} Suppose  $K(t)$ is the isotropic Mat\'ern autocovariance function defined in (\ref{DefineK}) for fixed parameters $\sigma, \alpha,\nu>0$.
 Then for any integer $m>\nu$, nonzero vector $\bs h\in \Bbb R^d$, matrix $M\in GL(d,\Bbb R)$ and bounded set $\Omega\subset \Bbb R^d$
\begin{equation}
\label{lemmaIneq1}
 \bigl|\partial_{\bs h}^{(m,m)} \bigl[K(|M\bs s-M\bs t|) \bigr]\bigr|  \lesssim |\bs s-\bs t|^{2\nu-2m}
\end{equation}
for all $\bs s,\bs t\in \Omega$ such that $\bs s\neq \bs t$.
\end{lemma}

\begin{proof}  First notice that it is sufficient to show the claim when $M$ is the identity matrix and $\alpha=1$ (extending to general $M$ and $\alpha$ follows by the chain rule for derivatives).
Define $K_{sq}(t)\triangleq K(\sqrt{t})$ and $F(\bs s,\bs t)\triangleq |\bs s-\bs t|^2$ so that
$\partial_{\bs h}^{(m,m)} \bigl[K(|\bs s-\bs t|)\bigr] = \partial_{\bs h}^{(m,m)} \bigl[K_{sq}(F(\bs s,\bs t))\bigr]$.
Also let $\partial_{\bs h}^*$ denote a generic directional derivative on either the variable $\bs s$ or $\bs t$.  
By generic I mean that $(\partial_{\bs h}^*)^k F$ denotes $\partial_{\bs h}^{(i,j)}F$ for some $i+j=k$  and $(\partial_{h}^* F)^k=\partial_{\bs h}^* F\cdots \partial_{\bs h}^* F$ where each $\partial_{\bs h}^*$ could be with respect to $\bs s$ or $\bs t$.
Now by successive application of the directional derivatives $\partial_{\bs h}^*$ we get that  
\begin{equation}
\label{sumOfS}
\partial_{\bs h}^{(m,m)} \bigl[K_{sq}(F(\bs s,\bs t))\bigr]=\sum_{i=1}^{2m} \sum_{\substack{0\leq j\leq i \\j+i\leq 2m }}K_{sq}^{(i)}(F(\boldsymbol s,\boldsymbol t)) (\partial_{\bs h}^*F(\boldsymbol s,\boldsymbol t))^{i-j} B_{ij} 
\end{equation}
where each $B_{ij}$ is uniformly bounded on $\Omega^2$. 
The functions $B_{ij}$ are uniformly bounded   by the nice fact that $(\partial_{\bs h}^*)^k F(\bs s,\bs t)\lesssim 1$ on $\Omega^2$ when $k\geq 2$.

We will bound the terms of the sum  (\ref{sumOfS})  when $i<\nu$, $i>\nu$, and $i=\nu$ separately.
Notice first that since $i\geq j$ we have that
\begin{align}
\label{uuu}
\bigl|\partial_{\bs h}^* F(\bs s,\bs t)\bigr|^{i-j}&\lesssim |\bs s-\bs t|^{i-j},\quad\text{for all $\bs s,\bs t\in \Omega$.} 
\end{align}
This implies, by Lemma \ref{ksqrt}, that the terms in the sum  (\ref{sumOfS}),  for which $i<\nu$, are bounded. 
When $i> \nu$
\begin{align*}
|K_{sq}^{(i)}(F(\bs s,\bs t)) (\partial_{\bs h}^*F(\bs s,\bs t))^{i-j} B_{ij}|
&\lesssim     |F(\bs s,\bs t)|^{\nu-i}|\bs s-\bs t|^{i-j},\quad\text{by (\ref{uuu}) and  Lemma \ref{ksqrt}}  \\
&=  |\bs s-\bs t|^{2\nu-(i+j)}   \\
&\lesssim   |\bs s-\bs t|^{2\nu-2m},\quad\text{since $i+j\leq 2m$,}
\end{align*}
where the inequality holds for all $\bs s,\bs t\in \Omega$ such that $|\bs s-\bs t|>0$ (note that we use the fact that $\Omega$ is bounded implies $|\bs s-\bs t|<T$ for some $T$).
For the last case, $i= \nu$,   a similar argument establishes
\begin{align*}
 |K_{sq}^{(i)}(F(\bs s,\bs t)) (\partial_{\bs h}^*F(\bs s,\bs t))^{i-j} B_{ij}|&\lesssim |\bs s-\bs t|^{i-j} |\log F(\bs s,\bs t)| \\
 &\lesssim |\bs s-\bs t|^{2\nu-2m} 
 \end{align*}
for all $\bs s,\bs t\in \Omega$ such that $|\bs s-\bs t|>0$.
Therefore  $ \partial_{\bs h}^{(m,m)} \bigl[K_{sq}(F(\bs s,\bs t))\bigr]\lesssim  |\bs s-\bs t|^{2\nu-2m}$
for all $\bs s,\bs t\in \Omega$ such that $\bs s\neq \bs t$.
\end{proof}

%-------------------------------------------------------
\begin{lemma} \label{suprem} Let $\bs h$ be a nonzero vector in $\Bbb R^d$,  $\nu>0$ and $H$ be the $d\times m$ matrix defined by  
\begin{equation}
\label{DefineH}
H\triangleq\underbrace{(\bs h,\cdots,\bs h)}_{\text{$m$ columns}}.
\end{equation}
 If $m$ is a positive integer greater than $\nu$ then  
\[\sup_{\bs {\xi,\eta} \in [0,1]^m} |\bs {i-j} + H (\bs \xi - \bs \eta)/n |^{2\nu-2m}\lesssim |\bs {i-j}|^{2\nu-2m} \]
for all positive integers $n$ and $\bs i,\bs j\in \Omega_n$ such that $|\bs i-\bs j|> |(m+1)\bs h/n|$.
\end{lemma}

\begin{proof} First notice that $\sup_{\bs {\xi,\eta} \in [0,1]^m}|\bs {i-j} + H (\bs \xi - \bs \eta)/n |^{2\nu-2m} = \sup_{-1\leq \tau \leq 1}|\bs {i-j} + m \bs h \tau/n |^{2\nu-2m}$. Now  for any $-1\leq \tau\leq 1$, positive integer $n$ and $\bs i,\bs j\in \Omega_n$ such that $|\bs i-\bs j|> |(m+1)\bs h/n|$, we have
\begin{align*}
|\bs i-\bs j+ m\bs h\tau/n|&\geq |\bs i-\bs j|-m|\tau||\bs h|/n\\
&\geq |\bs i-\bs j|-|\bs i-\bs j|\frac{m}{m+1}.
\end{align*}
The last line follows from the assumption that $|\bs i-\bs j|> (m+1)|\bs h|/n$ which implies $\frac{m}{m+1}|\bs i-\bs j|> m|\bs h|/n$.
Therefore
\[ \sup_{\bs {\xi,\eta} \in [0,1]^m}\bigl|\bs {i-j} + H (\bs \xi - \bs \eta)/n \bigr|^{2\nu-2m} \leq \Bigl(1-\frac{m}{m+1}\Bigr)^{2\nu-2m} |\bs {i-j}  |^{2\nu-2m}. \]
\end{proof}

%%%%%%%%%%%%%%%%%%%%%%%%
\begin{lemma}
\label{bounds1} 
Let $Y$ be a mean zero, geometric anisotropic $d$-dimensional Mat\'ern Gaussian random field with parameters $(\sigma, \alpha, \nu, M)$ and let $\Omega$ be a bounded, open subset of $\Bbb R^d$.
 Fix a positive integer $m>\nu$ and a non-zero vector $\bs h\in \Bbb R^d$.  Let $\Sigma$ to be the covariance matrix of the increments $ \Delta_{\bs h/n}^m Y(\bs i)$ as $\bs i$  ranges in the set $\bs i\in \Omega_n$ so that 
\begin{equation} 
\label{covmat}
\Sigma(\bs i,\bs j)\triangleq \mathsf E\bigl(  \Delta_{\bs h/n}^m Y(\bs i)  \Delta_{\bs h/n}^m Y(\bs j)\bigr)\end{equation}
for all $\bs i,\bs j\in \Omega_n$.
  Then there exists an $N>0$ such that   
\begin{equation}
   |\Sigma(\bs i,\bs j)|\lesssim n^{-2m} |\bs i-\bs j|^{2\nu-2m}
   \end{equation}
   for all $n>N$, and $\bs i,\bs j\in\Omega_n$ such that  $|\bs i-\bs j|> |(m+1)\bs h/n|$. Moreover,
   \begin{equation}
 |\Sigma(\bs i,\bs j)|\lesssim n^{-2\nu}
 \end{equation}
for all $n>N$ and   $\bs i,\bs j\in\Omega_n$ such that  $|\bs i-\bs j|\leq |(m+1)\bs h/n|$.
\end{lemma}

\begin{proof} First notice that $\Sigma(\bs i,\bs j)=\mathsf E \Delta_{\bs h/n}^m Y(\bs i) \Delta_{\bs h/n}^m Y(\bs j)= \Delta_{\bs h/n}^{(m,m)} K(|M(\bs i-\bs j)|)$ where $K$ is the isotropic Mat\'ern autocovariance function defined in (\ref{DefineK}). To simplify the notation let $F(\bs i,\bs j)\triangleq K(|M(\bs i-\bs j)|)$ and $H$ be the $d$ by $m$ matrix defined in (\ref{DefineH}). An induction argument on $m$ establishes that when $|\bs i-\bs j|> (m+1)|\bs h|/n$ we can express directional increments as  integrals of directional derivatives so that
\[\Delta_{\bs h/n}^{(m,m)} F(\bs i,\bs j)= \frac{1}{n^{2m}}\int_{\bs\xi,\bs \eta \in [0,1]^m}(\partial_{\bs h}^{(m,m)} F) (\bs i+ H\bs \xi/n, \bs j + H\bs \eta/n) d\bs \xi d\bs \eta.\]
Therefore
\begin{align*}
|\Sigma(\bs i,\bs j)|
%&=  \bigl|\mathsf E \Delta_{\bs h/n}^m Y(\bs i) \Delta_{\bs h/n}^m Y(\bs j)\bigr|\\
%&=\Bigl|  \frac{1}{n^{2m}} \int_0^1\cdots \int_0^1(\partial_{\bs h}^{(m,m)} K) (\bs i + H \bs \xi /n,\bs j+ H\bs \eta/n) d\bs \xi d\bs \eta \Bigr|  \\
&\lesssim \frac{1}{n^{2m}} \int_{\bs\xi,\bs\eta \in [0,1]^m}  \bigl| (\partial_{\bs h}^{(m,m)} F) (\bs i+ H\bs \xi/n, \bs j + H\bs \eta/n)  \bigr|  d\bs \xi d\bs \eta \\
&\lesssim \frac{1}{n^{2m}} \int_{\bs\xi,\bs \eta \in [0,1]^m}  \bigl|\bs {i-j} + H (\bs \xi - \bs \eta)/n \bigr|^{2\nu-2m} d\bs \xi d\bs \eta, \quad\text{by Lemma \ref{boundonderiv}} \\
&\lesssim \frac{1}{n^{2m}} \sup_{\bs {\xi,\eta} \in [0,1]^m}\bigl|\bs {i-j} + H (\bs \xi - \bs \eta)/n \bigr|^{2\nu-2m}\\
&\lesssim  \frac{1}{n^{2m}}\bigl |\bs {i-j}\bigr|^{2\nu-2m}\!,\quad\text{by Lemma \ref{suprem}}
\end{align*}
for all $n>N$, $\bs i,\bs j\in \Omega_n$ such that  $|\bs i-\bs j|> |(m+1)\bs h/n|$.
On the other hand when  $|\bs i-\bs j|\leq |(m+1)\bs h/n|$ 
\begin{align}
\label{eeiik}
|\Sigma(\bs i,\bs j)|&\leq \sqrt{\mathsf E(\Delta_{\bs h/n}^m Y(\bs i))^2}\sqrt{ \mathsf E(\Delta_{\bs h/n}^m Y(\bs j))^2} 
\lesssim n^{-2\nu}
\end{align}
where the last inequality is by Lemma  \ref{FirstLemma}.
 Actually, a direct application of Lemma \ref{FirstLemma} only establishes  (\ref{eeiik}) when $m>\nu+1$. However, a small adjustment  of the proof of Lemma \ref{FirstLemma} establishes that $\mathsf E(\Delta_{\bs h/n}^{m} Y(\bs t))^2 = \frac{a_\nu^{m}}{n^{2\nu}}+o(n^{-2\nu})$ as $n\rightarrow \infty$ when $m>\nu$. This is then is sufficient to establish (\ref{eeiik}). 
 %We also mention that to establish (\ref{eeiik}) we also use the uniformity of the convergence of $\mathsf E(\Delta_{\bs h/n}^{m} Y(\bs t))^2$ (obtained by stationarity) and the fact that $a_\nu^m>0$ (by Lemma \ref{Pozz}).
\end{proof}

%------------------------------------
\begin{lemma} 
\label{lalalemma}
Let $\Sigma_\text{abs}$ be the component-wise absolute value of the covariance matrix $\Sigma$ (defined in (\ref{covmat})). Then under the same assumptions as in Lemma \ref{bounds1}, there exits an $N>0$ such that
\[\|\Sigma_{abs}\|_2\lesssim  n^{-2\nu} + c\,{n^{d-2m}}\int_{1/n}^1 r^{2\nu-2m+d-1}dr.  \]
for all $n>N$, where $c$ is a constant and $\|\cdot\|_2$ is the spectral norm. \end{lemma}

\begin{proof}
First note that by symmetry, $\| \Sigma_{abs}\|_2\leq\sqrt{ \| \Sigma_{abs} \|_1\|\Sigma_{abs}\|_\infty}=\|\Sigma_{abs}\|_\infty$, where $\| \Sigma_{abs} \|_\infty$ is the maximum of the $\ell_1$ row norms and $\| \Sigma_{abs} \|_1$ is the maximum of the $\ell_1$ column norms. To bound the $\ell_1$ row norms, we bound the terms of the sum when $|\bs i-\bs j|>(m+1)|\bs h|/n$ and $|\bs i-\bs j|\leq (m+1)|\bs h|/n$ separately. 
For the off-diagonal terms we use Lemma \ref{bounds1} to ensure the existence of an $N>0$ such that for all $n>N$
\begin{align}
\max_{\bs i\in\Omega_n}\sum_{\substack{\bs j\in \Omega_n \\  |\bs i-\bs j|>(m+1)|\bs h|/n}}|\Sigma(\bs i,\bs j)| &\lesssim \max_{\bs i\in\Omega_n}\sum_{\substack{\bs j\in \Omega_n \\|\bs i-\bs j|>(m+1)|\bs h|/n}} n^{-2m} |\bs i-\bs j|^{2\nu-2m}  \label{3ee}\\
&\lesssim {n^{d-2m}} \int_{1/n}^1  r^{2\nu-2m+d-1} dr  \label{ee3}.
\end{align}
The last inequality, (\ref{ee3}), follows by the fact that for any constant $a>0$ and  open set $\Theta\subset \Bbb R^d$ which is bounded and contains the origin, one has 
\begin{equation} 
\label{sumtoint}
\sum_{\substack{\bs i\in\Theta\cap \{\Bbb Z^d/n\} \\ |\bs i|>a/n }} n^{-d} |\bs i|^\beta \lesssim  \int_{1/n}^{1} r^{\beta+d-1}dr
\end{equation}
 as $n\rightarrow\infty$ (for details see \cite{Anderes:thesis}, Lemma 3, page 41). In addition, by Lemma  \ref{bounds1}
\begin{align}
\max_{\bs i\in\Omega_n}\sum_{\substack{\bs j\in \Omega_n \\ |\bs i-\bs j|\leq (m+1)|\bs h|/n}} |\Sigma(\bs i,\bs j)| &\lesssim    n^{-2\nu}   \label{ee4} 
\end{align}
for all $n>N$.  This establishes the proof by noticing that the sum of the last terms in (\ref{ee3}) and (\ref{ee4}) bound $\| \Sigma_{abs} \|_\infty$. 
\end{proof}

%-----------------------------
\begin{lemma}
\label{boundd}
Under the same assumptions as in Lemma \ref{bounds1} there exits an $N>0$ such that 
\[ \|\Sigma\|_F^2 \lesssim {n^{d-4\nu}} + c\, n^{2d-4m}\int_{1/n}^1 r^{4\nu-4m+d-1}dr  \]
for all $n>N$ where $c$ is a constant and $\|\cdot\|_F$ denotes the Frobenious matrix norm. 
% Moreover, $\|\Sigma\|_F^2\asymp n^{d-4\nu}$ when $4(\nu-m)<-d$. 
\end{lemma}

\begin{proof} First note that $\|\Sigma\|_F^2=\sum_{\bs i,\bs j\in\Omega_n}|\Sigma(\bs i,\bs j)|^2$.  As in the proof of Lemma \ref{boundd} we bound the near-diagonal terms of $\Sigma$ separately from the off-diagonal terms.
By Lemma \ref{bounds1} there exists an $N>0$ such that 
\begin{align}
\sum_{\substack{\bs i,\bs j\in \Omega_n\\  |\bs i-\bs j|>(m+1)|\bs h|/n }} |\Sigma(\bs i,\bs j)|^2&\lesssim  n^{2d-4m} \sum_{\substack{\bs i,\bs j\in \Omega_n\\  |\bs i-\bs j|>(m+1)|\bs h|/n  }} n^{-2d} |\bs i-\bs j|^{4\nu-4m}\\
&\lesssim   n^{2d-4m}\int_{1/n}^1  r^{4\nu-4m+d-1} dr \label{ee1}
\end{align}
for all $n>N$. Notice that the last inequality is a slight variation on (\ref{sumtoint}).  
For the near diagonal terms we also use Lemma \ref{bounds1} to get
\begin{align}
\sum_{\substack{\bs i,\bs j\in \Omega_n  \\ |\bs i-\bs j|\leq (m+1)|\bs h|/n }} |\Sigma(\bs i,\bs j)|^2&\lesssim    n^d n^{-4\nu}. \label{ee2} 
\end{align}
Adding (\ref{ee1}) and (\ref{ee2}) establishes the lemma.
%To show that $\|\Sigma\|_F^2\asymp n^{d-4\nu}$, when $4\nu-4m<-d$ notice first that $\| \Sigma\|^2_F\lesssim n^{d-4\nu}$. Also
%\[n^{-d+4\nu} \| \Sigma\|^2_F\geq n^{-d+4\nu}\sum_{\bs i\in\Omega_n} \Sigma(\bs i,\bs i)^2= n^{-d+4\nu}\sum_{\bs i\in\Omega_n}\bigl( \mathsf E \Delta_{\bs h/n}^m Y(\bs i)^2  \bigr)^2= (a_\nu^m)^2 + O(n^{-2})\]
%by Lemma \ref{FirstLemma}. Noticing $0<a_\nu^m<\infty$ establishes the lemma.
\end{proof}

%-------------------------------------------------------
\begin{proof}[{\rm \bf Proof of Theorem 1}]
Define the random vector $\Delta Y$ to be the vector of $m$-increments, the components of which are indexed by $\Omega_n$ (in any order), so that
\begin{equation}
\label{vecdef}
\Delta Y\triangleq \underbrace{\bigl(\ldots,  \Delta_{\bs h/n}^m Y(\bs j),\ldots\bigr)}_{\text{terms are indexed by $\bs j\in\Omega_n$}}.
\end{equation}
Now we can write $Q_n^m=\frac{n^{2\nu}}{\# \Omega_n} \Delta Y \Delta Y^T=\frac{n^{2\nu}}{\# \Omega_n} W \Sigma W^T$, where $W\sim \mathcal N (\bs 0, I)$ (note that $\Sigma$ is defined in  (\ref{covmat})).
Therefore $\text{var}\, Q_n^m=2\frac{n^{4\nu}}{(\#\Omega)^2}\|\Sigma \|^2_F$ and by Lemma \ref{boundd}
\begin{align*}
\frac{n^{4\nu}}{(\#\Omega)^2}\|\Sigma \|^2_F&\lesssim n^{-d}+ c\, n^{4\nu-4m} \int_{1/n}^1 r^{4\nu-4m+d-1}dr \\
& \lesssim \begin{cases}
n^{4(\nu-m)},& \text{if $4(\nu-m)>-d$;} \\
n^{-d}\log n,& \text{if $4(\nu-m)=-d$;} \\
n^{-d}, &  \text{if $4(\nu-m)<-d$}
\end{cases} 
\end{align*}
for all sufficiently large $n$. This establishes the  variance rates. 

For the almost sure convergence result  let  $\tilde\Sigma\triangleq\frac{n^{2\nu}}{\# \Omega_n}\Sigma_{abs}$ where $\Sigma_{abs}$ is the component-wise absolute value of $\Sigma$. The Hanson and Wright bound in \cite{hanson:bound} then gives
\begin{equation}
 \mathsf P(|Q_n^m-EQ_n^m|\geq \epsilon)\leq 2\exp\left(-\frac{c_1\epsilon}{\| \tilde\Sigma \|_2}\wedge\frac{c_2\epsilon^2}{\| \tilde\Sigma\|_F^2}  \right)
 \end{equation}
 for all $\epsilon>0$, where $c_1,c_2$ are positive constants  not depending on $n$ or $\tilde\Sigma$.
First notice that by Lemma  \ref{lalalemma} we get
\begin{align}
\| \tilde\Sigma\|_2=\frac{n^{2\nu}}{\# \Omega_n}  \|  \Sigma_{abs}   \|_2&\lesssim n^{-d} + c\,{n^{2\nu-2m}}\int_{1/n}^1 r^{2\nu-2m+d-1}dr  \\
&\lesssim
\label{ratess} 
 \begin{cases}
n^{2(\nu-m)},& \text{if $2(\nu-m)>-d$;} \\
n^{-d}\log n,& \text{if $2(\nu-m)=-d$;} \\
n^{-d}, &  \text{if $2(\nu-m)<-d$.}
\end{cases}
\end{align}
for sufficiently large $n$. Also notice that this implies that $\|\tilde\Sigma\|_F^2\lesssim \|\tilde\Sigma\|_2$ for sufficiently large $n$. Therefore for sufficiently small $\epsilon$, $\mathsf P(|Q_n^m-EQ_n^m|\geq \epsilon)\leq 2\exp\left(-c_2\epsilon^2/ \|\tilde\Sigma\|_2  \right)$.  Now the rates in (\ref{ratess}) and the Borel-Cantelli Lemma are sufficient to establish that  $Q_n^m-\mathsf E Q_n^m\rightarrow 0$, with probability one as $n\rightarrow \infty$. 
 By Lemma \ref{FirstLemma}, $\mathsf E Q_n^m\rightarrow a_\nu^m$ (a slight adjustment also proves the case when $m>\nu$ rather than $m>\nu+1$) which establishes the theorem.
\end{proof}

%%%%%%%%%%%%%%%%%%%%%%%%%%%%%%%

\begin{proof}[\rm \bf Proof of Theorem 2]
First notice that when $p,q>\nu+1$
\begin{equation}
\label{eeeps}
 \mathsf E n^2\left[Q^p_n -\frac{a^p_\nu}{a^q_\nu}Q^q_n\right]\rightarrow  \left[b^p_\nu-  \frac{a^p_\nu}{a^q_\nu} b^q_\nu\right] \end{equation}
as $n\rightarrow \infty$ by Lemma \ref{FirstLemma}.
To get almost sure convergence notice
\begin{align}
\mathsf P\Bigl( n^2\Bigl|\bigr[Q^p_n-\frac{a^p_\nu}{a^q_\nu}Q^q_n \bigr] &- \mathsf E\bigl[Q^p_n -\frac{a^p_\nu}{a^q_\nu}Q^q_n\bigr]\Bigr|\geq \epsilon\Bigr) \\
\label{bfTerm}
&\leq \mathsf P\bigl(|[Q^p_n-\mathsf EQ^p_n|\geq \epsilon/2 n^2\bigr) + \mathsf P\bigl(|{a^p_\nu}||Q^q_n-\mathsf EQ^q_n|\geq |{a^q_\nu}| \epsilon/2 n^2 \bigr)
\end{align}
We can again use the Hanson and Wright bound (\cite{hanson:bound}) and the rates derived in Theorem \ref{FristThem} to get 
\begin{equation}
\label{LastInnq}
 \mathsf P(|Q_n^p-\mathsf EQ_n^p|\geq \epsilon/2n^2)\leq 2\exp\left(-c\,n^{-4}\epsilon^2/\|\tilde\Sigma\|_2  \right)
 \end{equation}
for all sufficiently small $\epsilon>0$ where $c$ is a positive constant that doesn't depend on $n$ or $\tilde\Sigma$. By inspection of the rates in (\ref{ratess}) the Borel-Cantelli Lemma can be applied when $4<\min\{2p-2\nu,d\}$ so that  $Q_n^p-EQ_n^p\rightarrow 0$ with probability one as $n\rightarrow \infty$.  A similar result holds for the second term in (\ref{bfTerm}) using the fact that both $a_\nu^p$ and $a_\nu^q$ are non-zero by Lemma \ref{Pozz}. This, combined with convergence of the expectation in (\ref{eeeps}), completes the proof.
\end{proof}
%%%%%%%%%%%%%%%%%%%%%%%%%%%

\begin{proof}[\rm \bf Proof of Theorem 4]
First notice that for any $p> \delta_2/2$  
 \begin{align}
 \label{beyond1proof}
\mathsf E(\Delta_{\bs h/n}^{p} Y(\bs t))^2 &=  \Delta_{\bs h/n}^{(p,p)} \Bigl\{c_1 |\bs x-\bs y|^{\delta_1}  + c_2 |\bs x-\bs y|^{\delta_2}  \Bigr\}_{\bs x=\bs y=\bs t} \\
%&= c_1|\bs h/n|^{\delta_1}\sum_{i,j=0}^p (-1)^{i+j}{p\choose i}{p\choose j} |i-j|^{\delta_1} + c_2|\bs h/n|^{\delta_2}\sum_{i,j=0}^p (-1)^{i+j}{p\choose i}{p\choose j} |i-j|^{\delta_2} \\
&= \frac{c_1 C_{p,\delta_1} }{n^{\delta_1}}+  \frac{c_2 C_{p,\delta_2}}{n^{\delta_2}}
\end{align}
 where 
$C_{p,\delta} \triangleq  |\bs h|^{\delta}\sum_{i,j=0}^p (-1)^{i+j}{p\choose i}{p\choose j} |i-j|^{\delta} $.  This follows since $Y$ is an intrinsic random function  of order  $\lfloor \delta_2/2 \rfloor$ with generalized autocovariance function $c_1|\cdot|^{\delta_1}+c_2|\cdot|^{\delta_2}$ and $\Delta_{\bs h/n}^p$ is an allowable linear combination of order $\lfloor \delta_2/2\rfloor$ (see \cite{chiles:book}). 
Now $Q_n^p$ is defined as in (\ref{DefofQQ}) with $\delta_1$ in place of $2\nu$ so that
\[
 Q_n^p \triangleq  \frac{1}{\#\Omega_n} \sum_{\bs j\in\Omega_n} n^{\delta_1}(\Delta_{\bs h/n}^p Y(\bs j))^2.
\]
For any integer $q,p>\delta_2/2$   we have that
$\mathsf E Q_n^p = c_1 C_{p,\delta_1}+ c_2 C_{p,\delta_2} n^{\delta_1-\delta_2}$ and 
$ n^{\delta_2-\delta_1}\mathsf E\bigl[  Q_n^p - \frac{C_{p,\delta_1}}{C_{q,\delta_1}} Q_n^q  \bigr]= c_2\bigl[  C_{p,\delta_2} - \frac{C_{p,\delta_1}}{C_{q,\delta_1}} C_{q,\delta_2}  \bigr]$. By a proof similar to Lemma \ref{Pozz} one can show that for any $p>\delta_2/2$ there exists a $q>p$ such that  $C_{q,\delta_1}\neq 0$ and $C_{p,\delta_2} - \frac{C_{p,\delta_1}}{C_{q,\delta_1}} C_{q,\delta_2} \neq 0$ (this uses the fact that $\delta_1$, $\delta_2$ are not even integers). This motivates the following definition
\begin{align}
\hat c_1&\triangleq Q_n^p/C_{p,\delta_1}\\
\hat c_2&\triangleq n^{\delta_2-\delta_1}\frac{  Q_n^p - \frac{C_{p,\delta_1}}{C_{q,\delta_1}} Q_n^q }{ C_{p,\delta_2} - \frac{C_{p,\delta_1}}{C_{q,\delta_1}} C_{q,\delta_2} }.
\end{align}
In what follows we show that $p>\delta_2/2$ implies $\hat c_1 \overset{a.s.}\longrightarrow c_1$ as $n\rightarrow \infty$. Moreover if  $2(\delta_2-\delta_1)< \min\{ 2p-\delta_1,d\}$ then
 there exists a $q>p$ such that  $\hat c_2 \overset{a.s.}\longrightarrow c_2$.

We start by letting $ \Sigma(\bs i,\bs j)\triangleq \mathsf E\bigl(  \Delta_{\bs h/n}^p Y(\bs i)  \Delta_{\bs h/n}^p Y(\bs j)\bigr)$ for all $\bs i,\bs j\in \Omega_n$ and  $\tilde\Sigma\triangleq\frac{n^{\delta_1}}{\# \Omega_n}\Sigma_{abs}$ where $\Sigma_{abs}$ is the component-wise absolute value of $\Sigma$. 
The Hanson and Wright bound in \cite{hanson:bound} gives
\begin{equation}
\label{HWagain}
 \mathsf P(|Q_n^p-EQ_n^p|\geq \epsilon)\leq 2\exp\left(-\frac{b_1\epsilon}{\| \tilde\Sigma \|_2}\wedge\frac{b_2\epsilon^2}{\| \tilde\Sigma\|_F^2}  \right)
 \end{equation}
 for all $\epsilon>0$, where $b_1,b_2$ are positive constants  not depending on $n$ or $\tilde\Sigma$.
Later in the proof we will show  that when $p>\delta_2/2$
\begin{equation}
\label{eeennn}
\|\tilde\Sigma\|_F^2\lesssim \|\tilde\Sigma\|_2 \lesssim \begin{cases}
n^{\delta_1-2p},& \text{if $0>\delta_1-2p>-d$;} \\
n^{-d}\log n,& \text{if $\delta_1-2p=-d$;} \\
n^{-d}, &  \text{if $\delta_1-2p<-d$.}
\end{cases}
\end{equation} 
for all sufficiently large $n$. First, however, we show this is sufficient for the almost sure convergence result. Equations  (\ref{HWagain})  and (\ref{eeennn}) immediately establishes that $\hat c_1 \overset{a.s.}\longrightarrow c_1$ when $p>\delta_2/2$ since $|Q_n^p-EQ_n^p|\overset{a.s.}\longrightarrow 0$ by Borel-Cantelli and $\mathsf E \hat c_1\rightarrow c_1$.
% and 
%$ \mathsf P(|Q_n^p-EQ_n^p|\geq \epsilon)\leq 2\exp\left(-b_2\epsilon^2/\| \tilde\Sigma\|_2  \right)\leq 2\exp\left(-b_3\epsilon^2n^d \right)$ for all sufficiently small $\epsilon>0$ and all sufficiently large $n$ where $b_3$ is a positive constant.  
To see that $\hat c_2\overset{a.s.}\longrightarrow c_2$ notice that
\begin{align}
 \mathsf P(n^{\delta_2-\delta_1}|Q_n^p-EQ_n^p|\geq \epsilon)&\leq 2\exp\left(-\frac{b_2\epsilon^2}{n^{2(\delta_2-\delta_1)}\| \tilde\Sigma\|_2 } \right) 
 %\\
 %&\leq 2\exp\left(-b_3\epsilon^2n^{d- 2(\delta_2-\delta_1)  } \right)
\end{align}
for all sufficiently small $\epsilon>0$ and sufficiently large $n$. Therefore, by inspection of the rates in (\ref{eeennn}), one can show that  $n^{\delta_2-\delta_1}|Q_n^p-EQ_n^p|\overset{a.s.}\longrightarrow 0$ whenever $2(\delta_2-\delta_1)< \min\{ 2p-\delta_1,d\}$  and $p>\delta_2/2$. Since $\mathsf E \hat c_2=c_2$, this is sufficient to establish that there exists a $q>p$ such that $\hat c_2 \overset{a.s.}\longrightarrow c_2$ as $n\rightarrow \infty$ when $0<2(\delta_2-\delta_1)< \min\{ 2p-\delta_1,d\}$ and $p>\delta_2/2$.

Now to finish the proof we need to establish (\ref{eeennn}). We start by noticing that 
for any $p> \delta_2/2$ there exists a constant $c>0$ such that 
\begin{align*}
 \bigl|\partial_{\bs h}^{(p,p)} \text{cov}(Y(\bs s),Y(\bs t)) \bigr|  &=  \bigl|\partial_{\bs h}^{(p,p)}\bigl\{   c_1|\bs s-\bs t|^{\delta_1}+ c_2|\bs s-\bs t|^{\delta_2}\bigr\}\bigr| \\
 &\leq b_4 |\bs s-\bs t|^{\delta_1-2p}
\end{align*}
for all $\bs s,\bs t\in \Omega$ such that $\bs s\neq \bs t$ where $b_4$ is a positive constant.  Following the proofs of Lemma \ref{bounds1} we can then derive that  there exists an $N>0$ such that for any $p>\delta_2/2$ 
\[ |\Sigma(\bs i,\bs j)|\lesssim n^{-2p} |\bs i-\bs j|^{\delta_1-2p}\]
   for all $n>N$, and $\bs i,\bs j\in\Omega_n$ such that  $|\bs i-\bs j|> |(p+1)\bs h/n|$. Moreover,
\[ |\Sigma(\bs i,\bs j)|\lesssim n^{-\delta_1}\]
for all $n>N$ and   $\bs i,\bs j\in\Omega_n$ such that  $|\bs i-\bs j|\leq |(p+1)\bs h/n|$. Now by direct analogs to Lemma \ref{lalalemma}, Lemma \ref{boundd}  (by replacing $2\nu$ with $\delta_1$) we have 
\begin{align} 
\label{oness}
\|\tilde\Sigma\|_2 &\lesssim n^{-d} + b_5\, n^{\delta_1-2p} \int_{1/n}^1 r^{\delta_1-2p+d-1}dr \\
\label{twoss}
\|\tilde\Sigma\|_F^2 &\lesssim  n^{-d}+ b_6\, n^{2\delta_1-4p} \int_{1/n}^1 r^{2\delta_1-4p+d-1}dr  
\end{align}
where $b_5$, $b_6$ are positive constants. 
These equations hold when $p>\delta_2/2$ and are sufficient to establish (\ref{eeennn}).
This completes the proof. 
\end{proof}

%%%%%%%%%%%%%%%%%%%%%%%%%%%%%%%
\begin{proof}[\rm \bf Proof of Theorem 5] 
First notice that since $e^{-|\bs t|^\delta}= 1 -  |\bs t|^\delta + |\bs t|^{2\delta}/2+ O(|\bs t|^{3\delta})$ one can show that for any $p\geq 1$   
 \begin{align}
 \label{beyond1proof}
\mathsf E(\Delta_{\bs h/n}^{p} Y(\bs t))^2 &=  \Delta_{\bs h/n}^{(p,p)} \Bigl\{ \sigma^2\exp\left[ -\bigl|\alpha M\bs x-\alpha M\bs y\bigr|^{\delta} \right] \Bigr\}_{\bs x=\bs y=\bs t} \\
\label{vann}
 &=-\sigma^2 \alpha^\delta |M \bs h|^{\delta}\frac{ D_{p,\delta} }{n^{\delta}}+ \sigma^2\alpha^{2\delta}  |M \bs h|^{2\delta} \frac{D_{p,2 \delta}}{2n^{2\delta}} +O(n^{-3\delta})
\end{align}
 where 
$D_{p,\delta} \triangleq  \sum_{i,j=0}^p (-1)^{i+j}{p\choose i}{p\choose j} |i-j|^{\delta} $.  This follows by a proof exactly similar to that of Lemma \ref{FirstLemma}.

Now $Q_n^p$ is defined as in (\ref{DefofQQ}) with $\delta$ in place of $2\nu$ so that
\[
 Q_n^p \triangleq  \frac{1}{\#\Omega_n} \sum_{\bs j\in\Omega_n} n^{\delta}(\Delta_{\bs h/n}^p Y(\bs j))^2.
\]
Therefore
\begin{align}
\mathsf E Q_n^p &=-\sigma^2 \alpha^\delta |M \bs h|^{\delta} D_{p,\delta}+ O(n^{-\delta}) 
\end{align}
and
\begin{align} \mathsf En^{\delta}\bigl[  Q_n^p - \frac{D_{p,\delta}}{D_{q,\delta}} Q_n^q  \bigr]&= \frac{\sigma^2\alpha^{2\delta}  |M \bs h|^{2\delta}}{2}\bigl[  D_{p,2\delta} - \frac{D_{p,\delta}}{D_{q,\delta}} D_{q,2\delta}  \bigr]+O(n^{-\delta}).
 \end{align}
  By a proof similar to Lemma \ref{Pozz} one can show that for any $p\geq 1$ one has that $D_{q,\delta}\neq 0$ and if, additionally, $p>\delta$ there exists a $q>p$ such that and $D_{p,2\delta} - \frac{D_{p,\delta}}{D_{q,\delta}} D_{q,2\delta} \neq 0$ (this uses the fact that $2\delta$ is not an even integer). This motivates the following definition
\begin{align}
\label{imbd1}
\widehat{ \sigma^2 \alpha^\delta |M \bs h|^{\delta} }&\triangleq -Q_n^p/D_{p,\delta}\\
\label{imbd2}
\widehat{ \sigma^2 \alpha^{2\delta} |M \bs h|^{2\delta} }&\triangleq 2n^{\delta}\frac{  Q_n^p - \frac{D_{p,\delta}}{D_{q,\delta}} Q_n^q }{ D_{p,2\delta} - \frac{D_{p,\delta}}{D_{q,\delta}} D_{q,2\delta} }.
\end{align}
In what follows we show that for any $p\geq 1$ we have that $ \widehat{ \sigma^2 \alpha^\delta |M \bs h|^{\delta} }\overset{a.s.}\longrightarrow  \sigma^2 \alpha^\delta |M \bs h|^{\delta}$  as $n\rightarrow \infty$. Moreover, if
$2\delta<\min\{ 2p-\delta,d \}$ and $p>\delta$ (this is required to guarantee that $D_{p,2\delta}\neq 0$) there exists a $q>p$ such that  and $\widehat{ \sigma^2 \alpha^{2\delta} |M \bs h|^{2\delta} }\overset{a.s.}\longrightarrow  \sigma^2 \alpha^{2\delta} |M \bs h|^{2\delta} $.  Notice that this is sufficient to prove the theorem since $p>3\delta/2$ and $ 2\delta<d$ together imply that $2\delta<\min\{ 2p-\delta,d \}$ and $p>\delta$. 
%Now, by the discussion in Section \ref{Geoo}, the strong consistency of the estimates (\ref{imbd1}) and (\ref{imbd2}) allows one to consistenly estimate $\alpha$, $\sigma$ and $M$ using finitely many directions $\bs h$. 

We start by letting $ \Sigma(\bs i,\bs j)\triangleq \mathsf E\bigl(  \Delta_{\bs h/n}^p Y(\bs i)  \Delta_{\bs h/n}^p Y(\bs j)\bigr)$ for all $\bs i,\bs j\in \Omega_n$ and  $\tilde\Sigma\triangleq\frac{n^{\delta}}{\# \Omega_n}\Sigma_{abs}$ where $\Sigma_{abs}$ is the component-wise absolute value of $\Sigma$. 
The Hanson and Wright bound in \cite{hanson:bound} gives
\begin{equation}
\label{HWagain2}
 \mathsf P(|Q_n^p-EQ_n^p|\geq \epsilon)\leq 2\exp\left(-\frac{b_1\epsilon}{\| \tilde\Sigma \|_2}\wedge\frac{b_2\epsilon^2}{\| \tilde\Sigma\|_F^2}  \right)
 \end{equation}
 for all $\epsilon>0$, where $b_1,b_2$ are positive constants  not depending on $n$ or $\tilde\Sigma$.
Later in the proof we will show  that 
\begin{equation}
\label{eeennn2}
\|\tilde\Sigma\|_F^2\lesssim \|\tilde\Sigma\|_2 \lesssim  \begin{cases}
n^{\delta-2p},& \text{if $0>\delta-2p>-d$;} \\
n^{-d}\log n,& \text{if $\delta-2p=-d$;} \\
n^{-d}, &  \text{if $\delta-2p<-d$.}
\end{cases}
\end{equation} 
for all sufficiently large $n$. First, however, we show this is sufficient for the almost sure convergence result. Equations (\ref{HWagain2}) and (\ref{eeennn2})  immediately establishes that $ \widehat{ \sigma^2 \alpha^\delta |M \bs h|^{\delta} }\overset{a.s.}\longrightarrow  \sigma^2 \alpha^\delta |M \bs h|^{\delta}$ since $|Q_n^p-EQ_n^p|\overset{a.s.}\longrightarrow 0 $ by Borel-Cantelli and $\mathsf E \widehat{ \sigma^2 \alpha^\delta |M \bs h|^{\delta} }\rightarrow \sigma^2 \alpha^\delta |M \bs h|^{\delta}$. 
%and 
%$ \mathsf P(|Q_n^p-EQ_n^p|\geq \epsilon)\leq 2\exp\left(-b_2\epsilon^2/\| \tilde\Sigma\|_2  \right)\leq 2\exp\left(-b_3\epsilon^2n^d \right)$ for all sufficiently small $\epsilon>0$ and all sufficiently large $n$ where $b_3$ is a positive constant.  
To see that $\widehat{ \sigma^2 \alpha^{2\delta} |M \bs h|^{2\delta} }  \overset{a.s.}\longrightarrow  \sigma^2 \alpha^{2\delta} |M \bs h|^{2\delta}$ notice that equations  (\ref{HWagain2}) and (\ref{eeennn2}) imply
\begin{align}
 \mathsf P(n^{\delta}|Q_n^p-EQ_n^p|\geq \epsilon)&\leq 2\exp\left(-\frac{b_2\epsilon^2}{n^{2\delta}\| \tilde\Sigma\|_2 } \right) % \\
% &\leq 2\exp\left(-b_3\epsilon^2n^{d- 2\delta  } \right)
\end{align}
for all sufficiently small $\epsilon>0$ and sufficiently large $n$. Therefore $n^{\delta}|Q_n^p-EQ_n^p|\overset{a.s.}\longrightarrow 0$ whenever $2\delta< \min\{2p-\delta,d \}$. Since $\mathsf E \widehat{ \sigma^2 \alpha^{2\delta} |M \bs h|^{2\delta} } \rightarrow  \sigma^2 \alpha^{2\delta} |M \bs h|^{2\delta}$, this is sufficient to establish a $q>p>\delta$ so that $\widehat{ \sigma^2 \alpha^{2\delta} |M \bs h|^{2\delta} }  \overset{a.s.}\longrightarrow  \sigma^2 \alpha^{2\delta} |M \bs h|^{2\delta}$.

Now to finish the proof we need to establish (\ref{eeennn2}). We start by noticing that 
for any $p\geq 1$  there exists a constant $c>0$ such that 
\begin{align*}
 \bigl|\partial_{\bs h}^{(p,p)} \text{cov}(Y(\bs s),Y(\bs t)) \bigr|  &=  \bigl|\partial_{\bs h}^{(p,p)}\Bigl\{  \sigma^2\exp\left[ -\bigl|\alpha M\bs x-\alpha M\bs y\bigr|^{\delta} \right]    \Bigr\}\bigr| \\
 &\leq b_4 |\bs s-\bs t|^{\delta-2p}
\end{align*}
for all $\bs s,\bs t\in \Omega$ such that $\bs s\neq \bs t$ where $b_4$ is a positive constant.  Following the proofs of Lemma \ref{bounds1} we can then derive that  there exists an $N>0$ such that for any $p\geq 1$ 
\[ |\Sigma(\bs i,\bs j)|\lesssim n^{-2p} |\bs i-\bs j|^{\delta-2p}\]
   for all $n>N$, and $\bs i,\bs j\in\Omega_n$ such that  $|\bs i-\bs j|> |(p+1)\bs h/n|$. Moreover,
\[ |\Sigma(\bs i,\bs j)|\lesssim n^{-\delta}\]
for all $n>N$ and   $\bs i,\bs j\in\Omega_n$ such that  $|\bs i-\bs j|\leq |(p+1)\bs h/n|$. Now by direct analogs to Lemma \ref{lalalemma}, Lemma \ref{boundd}  (by replacing $2\nu$ with $\delta$) we have 
\begin{align} 
\label{oness2}
\|\tilde\Sigma\|_2 &\lesssim n^{-d} + b_5\, n^{\delta-2p} \int_{1/n}^1 r^{\delta-2p+d-1}dr \\
\label{twoss2}
\|\tilde\Sigma\|_F^2 &\lesssim  n^{-d}+ b_6\, n^{2\delta-4p} \int_{1/n}^1 r^{2\delta-4p+d-1}dr  
\end{align}
where $b_5$, $b_6$ are positive constants. 
These equations hold when $p\geq 1$ and establish (\ref{eeennn2}). This completes the proof. 
\end{proof}

\bibliography{references}
\end{document}